\documentclass[11pt,a4paper]{article}
\usepackage[utf8]{inputenc}
\usepackage{amsmath,amssymb,amsfonts}
\usepackage{enumerate}
\usepackage[margin=3cm]{geometry}
\hypersetup{
	colorlinks=true,        % false: boxed links; true: colored links
	linkcolor=blue,         % color of internal links
	citecolor=blue,         % color of links to bibliography
	urlcolor=blue           % color of external links
}
\usepackage{graphicx}
\usepackage{mathtools}
\usepackage[ruled]{algorithm2e}
\usepackage{array}
\usepackage{subcaption}
\usepackage{todonotes}

\mathtoolsset{showonlyrefs=true}

\usepackage{amsthm}
\usepackage{thmtools}
\newtheorem{theorem}{Theorem}[section]
\newtheorem{proposition}[theorem]{Proposition}
\newtheorem{corollary}[theorem]{Corollary}
\newtheorem{definition}[theorem]{Definition}

\newtheorem{lemma}[theorem]{Lemma}
\declaretheorem[style=remark,qed=$\Diamond$,Refname={Remark,Remarks}]{remark}
\declaretheorem[style=remark,qed=$\Diamond$,Refname={Example,Examples}]{example}

\DeclareMathOperator{\Id}{Id}

\DeclareMathOperator{\dom}{dom}
\DeclareMathOperator{\sign}{sign}
\DeclareMathOperator{\ran}{ran}
\DeclareMathOperator{\gra}{gra}
\DeclareMathOperator{\zer}{zer}
\DeclareMathOperator{\Fix}{Fix}
\DeclareMathOperator{\prox}{prox}
\DeclareMathOperator*{\argmin}{argmin}
\newcommand{\Hilbert}{\mathcal{H}}
\newcommand{\Hi}{\mathcal{H}}
\newcommand{\R}{\mathbb{R}}

\newcommand{\setto}{\rightrightarrows}
\newcommand{\wto}{\rightharpoonup}
\newcommand*\colvec[1]{\begin{pmatrix}#1\end{pmatrix}}

\title{A direct proof of convergence of Davis--Yin splitting algorithm allowing larger stepsizes}
\author{Francisco J.\ Arag\'on-Artacho\thanks{Department of Mathematics,
                             University of Alicante,
                             Alicante, \textsc{Spain}.
	                         Email:~\href{mailto:francisco.aragon@ua.es}
	                         {francisco.aragon@ua.es}}
          \and
          David Torregrosa-Bel\'en\thanks{Department of Mathematics,
                             University of Alicante,
                             Alicante, \textsc{Spain}.
	                         Email:~\href{mailto:david.torregrosa@ua.es}
	                         {david.torregrosa@ua.es}}}

\graphicspath{{Figures/}}

\begin{document}
\maketitle

\begin{abstract}
This note is devoted to the splitting algorithm proposed by Davis and Yin in 2017
for computing a zero of the sum of three maximally monotone operators, with one of them being cocoercive. %
We provide a direct proof that guarantees its convergence when the stepsizes are smaller than four times the cocoercivity constant, thus doubling the size of the interval established by Davis and Yin. As a by-product, the same conclusion applies to the forward-backward splitting algorithm. Further, we use the notion of ``strengthening'' of a set-valued operator to derive a new splitting algorithm for computing the resolvent of the sum. Last but not least, we provide some numerical experiments illustrating the importance of appropriately choosing the stepsize and relaxation parameters of the algorithms.
\end{abstract}


\section{Introduction}

When a problem has certain structure, it is normally useful to take advantage of it. Following the divide-and-conquer paradigm, splitting algorithms iteratively solve simpler problems which are defined by separately using some parts of the original problem. A particular subfamily are projection methods (see, e.g.,~\cite[Chapter~5]{Cegielski}), which can be used to find a common point in the intersection of sets, based on projections of points defined in the iterations into each of the sets. These methods are usually variations of classical iterative schemes for finding fixed points of certain type of nonexpansive operators. Monotone operator theory~\cite{bauschke2017} permits to generalize these algorithms to tackle the far more general problem of finding a zero of the sum of maximally monotone operators by using their resolvents instead of the projectors (see Definitions~\ref{def:monotonicity} and~\ref{def:resolvent}).

There are many different splitting algorithms for computing a zero of the sum of two maximally monotone operators (see, e.g.,~\cite[Chapter~26]{bauschke2017}). Theoretically, one can always transform any splitting algorithm for computing zeros of the sum of two operators into a splitting algorithm for computing zeros of the sum of finitely many operators (see, e.g.,~\cite[Proposition 26.4]{bauschke2017}), thanks to Pierra's product space reformulation~\cite{Pierra}. Nevertheless, numerical experience shows that this theoretical trick usually slows down the resulting algorithm (see, e.g.,~\cite[Section~6.1]{franstrengthening}), especially when the number of operators is large (see, e.g.,~\cite[Section~4]{YairAviv} and~\cite[Section~5]{BorweinTam2014}). To alleviate this problem, various schemes requiring one space less in the product space have been recently proposed~\cite{Campoy21,TamEtAl21,malitsky2021resolvent}.

Only recently, three-operator splitting algorithms have been developed~\cite{davis2017three,patrinos17,rieger2020backward,ryu2019uniqueness,RyuVu2020}. This note is devoted to one of them, which was introduced by Damek Davis and Wotao Yin in~\cite{davis2017three}, and is commonly referred as \emph{Davis--Yin splitting algorithm}. The algorithm is designed for solving the problem
\begin{equation}\label{eq:problem}
\text{find~}x \text{ such that } 0 \in (A+B+T)(x),
\end{equation}
where all three operators involved are maximally monotone and act on a Hilbert space, and $T$ is also cocoercive (see Definition~\ref{def:cocoercive}).
Davis and Yin defined the operator
\begin{equation}\label{eq:DYoperator}
DY_\gamma:=J_{\gamma B}\circ \left(2J_{\gamma A}-\Id-\gamma T\circ J_{\gamma A}\right)+\Id-J_{\gamma A},
\end{equation}
where $J_{\gamma A}$ and $J_{\gamma B}$ denote the corresponding resolvents, and proved that $DY_\gamma$ is $\alpha$-\emph{averaged} for $\alpha=\frac{2\beta}{4\beta-\gamma}$ when $\gamma\in{]0,2\beta[}$, where $\beta>0$ is the cocoercivity constant of~$T$. Then, they defined their splitting algorithm through the standard Krasnosel'ski\u{i}--Mann iteration
\begin{equation}\label{eq:DY0}
x_{k+1}=(1-\lambda_k)x_k+\lambda_k DY_\gamma(x_k),\quad k=0,1,2,\ldots,
\end{equation}
with $\lambda_k\in{]0,1/\alpha[}$ satisfying the assumptions of~\cite[Proposition~5.16]{bauschke2017}, from which its convergence to a fixed point $x$ of $DY_\gamma$ follows. Further, the \emph{shadow sequence} $(J_{\gamma A}(x_k))_{k\in\mathbb{N}}$ weakly converges to a solution to~\eqref{eq:problem}, and convergence is strong under additional assumptions. Three well-known splitting algorithms can be obtained as a particular instance of Davis--Yin's, namely the Douglas--Rachford~\cite{LionsMercier} (when $T=0$), the forward-backward~\cite{LionsMercier,Passty} (when $A=0$) and the backward-forward~\cite{Attouch18} (when $B=0$).

In this note we provide a direct proof of the convergence of the iterative method~\eqref{eq:DY0} without relying on the averagedness of the operator $DY_\gamma$ (see Theorem~\ref{t:GradientDR}). Our proof has two key advantages: (i) it permits to simplify the assumptions on the relaxation parameters, and (ii) it allows to choose the stepsize $\gamma$ in $]0,4\beta[$ instead of $]0,2\beta[$. Observe that the operator $DY_\gamma$ does not need to be averaged when $\gamma>2\beta$ (for instance, take $A=B=0$, $T$ the identity, and apply $DY_\gamma$ to the points $x=1$ and $z=-1$). As a by-product, this shows that the stepsize in the forward-backward and the backward-forward algorithms can be also chosen in $]0,4\beta[$. In addition, we derive in Theorem~\ref{t:DYsplitting} a \emph{strengthened} version of Davis--Yin splitting algorithm which permits computing the resolvent of $A+B+T$.

Right before submitting this manuscript, we learnt about the recent preprint~\cite{DaoPhan21}. Using the notion of \emph{conically averaged} operators introduced in~\cite{BartzDaoPhan20}, the authors prove in~\cite[Corollary~4.2]{DaoPhan21} that the operator $(1-\lambda)\Id+\lambda DY_\gamma$ is $2\lambda\beta/(4\beta-\gamma)$-averaged when $\gamma\in{]0,4\beta[}$, from which the convergence of~\eqref{eq:DY0} for a fixed $\lambda_k=\lambda$ follows.

As a simple motivating example of the importance of the algorithm parameters, consider the problem of finding the minimum norm point in the intersection of two balls $\mathbb{A}$ and $\mathbb{B}$ in the Euclidean space whose intersection has nonempty interior. The problem can be solved with Davis--Yin splitting algorithm, taking $A$ and $B$ as the normal cones to the respective balls, and $T$ as the identity mapping. Since the resolvents of the normal cones are the projectors (see Example~\ref{ex:2}), which we denote by $P_{\mathbb{A}}$ and $P_{\mathbb{B}}$, the iterative scheme is given by
$$x_{k+1}=x_k-\lambda_k P_\mathbb{A}(x_k)+\lambda_k P_\mathbb{B}\left((2-\gamma)P_\mathbb{A}(x_k)-x_k\right),\quad k=0,1,2,\ldots,$$
and $\left(P_{\mathbb{A}}(x_k)\right)_{k\in\mathbb{N}}$ converges to the minimum norm point in $\mathbb{A}\cap\mathbb{B}$ (the normal cone sum rule holds).
Both the \emph{relaxation parameter} $\lambda_k$ and the \emph{stepsize} $\gamma$ have a big influence on the behavior of the algorithm, as shown in Figure~\ref{fig:1}.
\begin{figure}[ht!]\centering
\includegraphics[width=.49\textwidth]{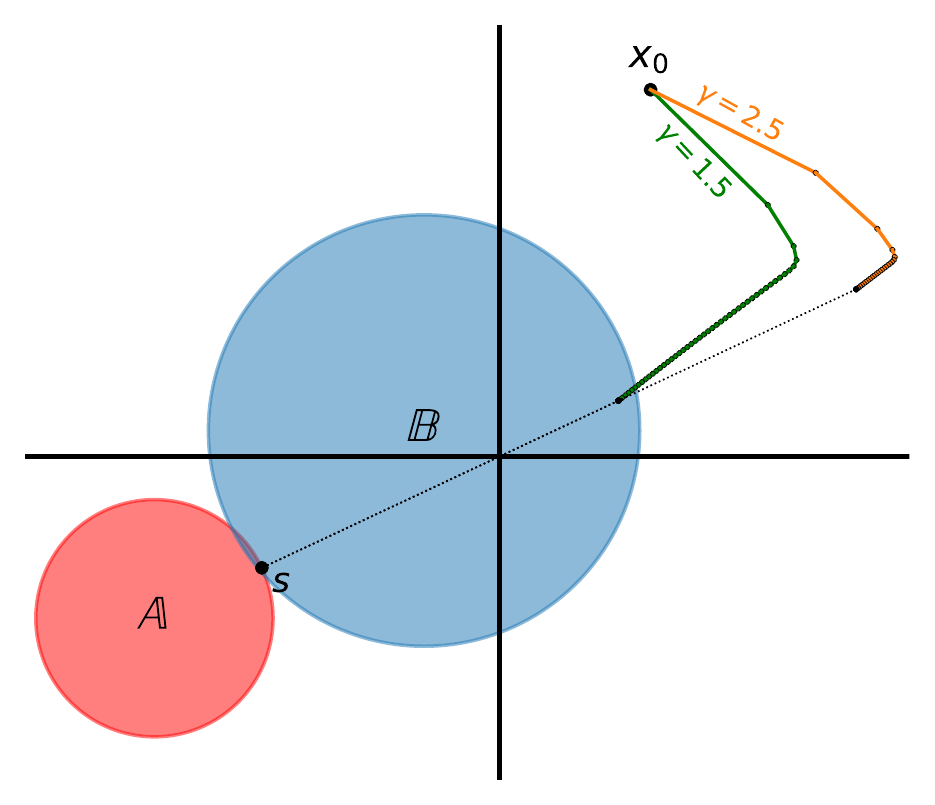}
\includegraphics[width=.49\textwidth]{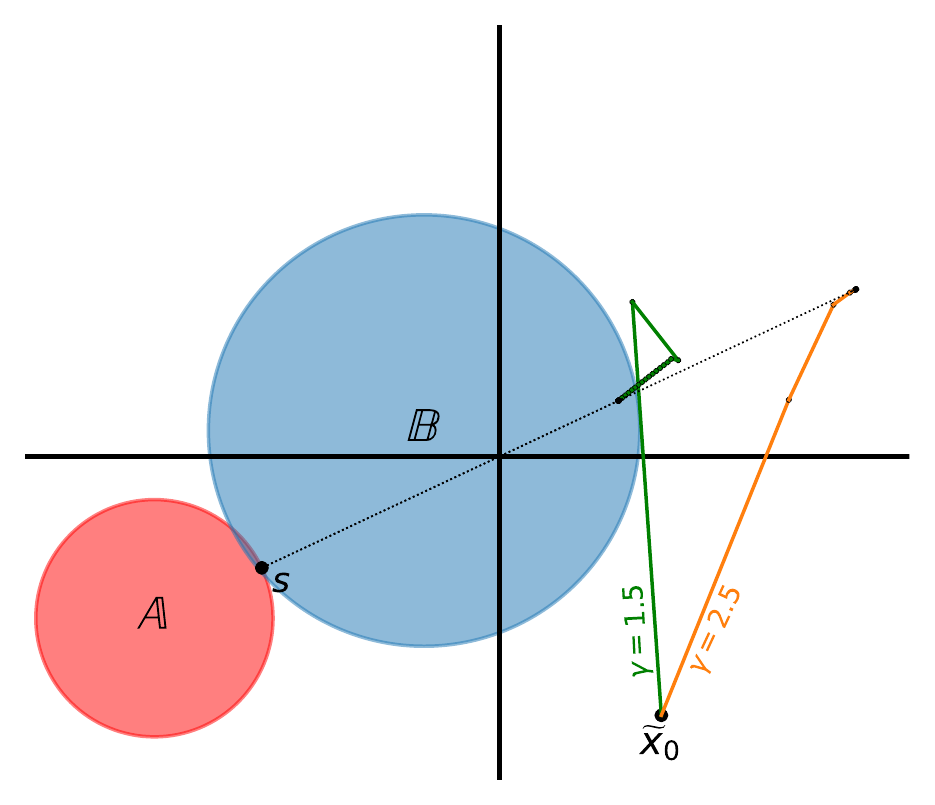}
\caption{Behavior of Davis--Yin splitting algorithm for two starting points $x_0$ and $\widetilde{x}_0$ and two stepsize parameters $\gamma$, with $\lambda_k=0.99(2-\gamma/2)$. The solution $s$ is obtained after projecting the fixed point onto $\mathbb{A}$.}\label{fig:1}
\end{figure}
In this example, since the cocoercivity constant $\beta$ is equal to $1$, \cite[Theorem~2.1]{davis2017three} guarantees the convergence when the parameter $\gamma$ is taken in $]0,2[$, while Theorem~\ref{t:GradientDR} allows to take $\gamma\in{]0,4[}$. When the Davis--Yin splitting algorithm is applied to the same problem with different starting points $x_0$, it can behave very differently depending on the parameters, as shown in Figures~\ref{fig:1} and~\ref{fig:2}.
\begin{figure}[ht!]\centering
\includegraphics[width=.49\textwidth]{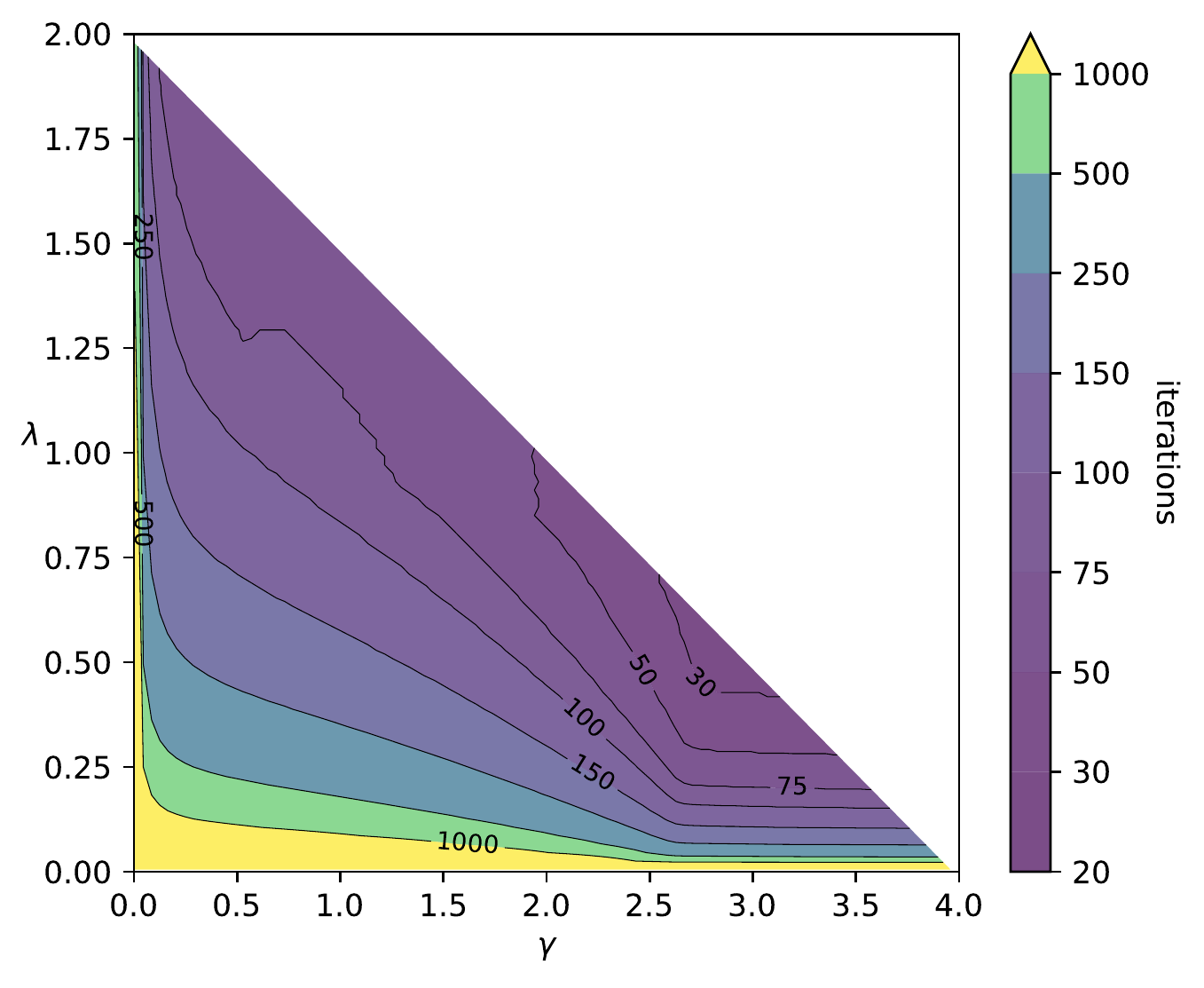}
\includegraphics[width=.49\textwidth]{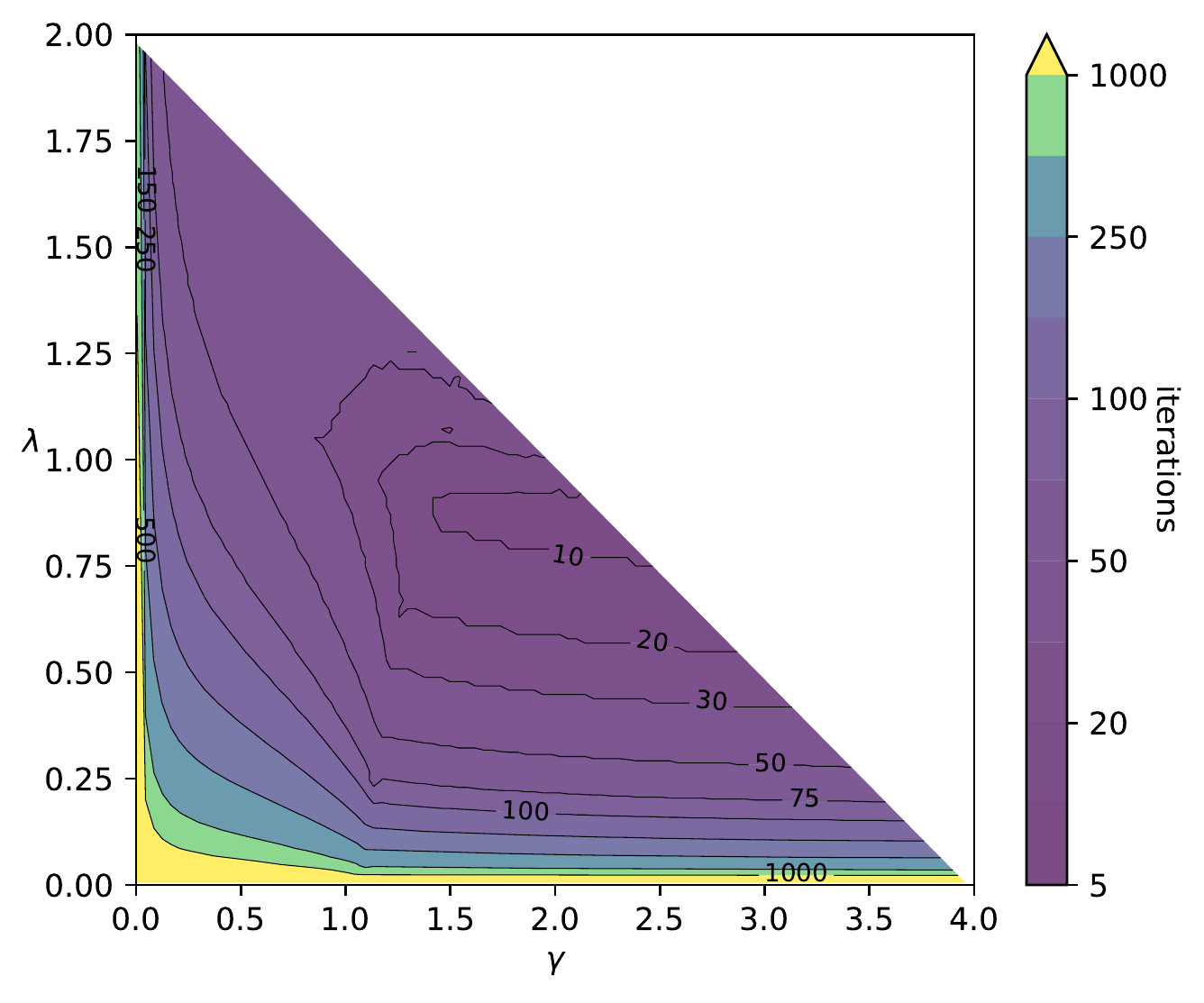}
\caption{Number of iterations needed until the shadow sequence gets sufficiently close to the solution $s$ (precisely, $\|P_A(x_k)-s\|<10^{-10}$) for different values of $\gamma$ and $\lambda_k=\lambda$, with starting points $x_0$ (left) and $\widetilde{x}_0$ (right) shown in Figure~\ref{fig:1}.}\label{fig:2}
\end{figure}

In general, larger stepsizes are commonly believed to be associated with faster convergence of algorithms, but this is not always the case, particularly when an algorithm has several parameters. It is important to have in mind that the relaxation parameter $\lambda_k$ of the Davis--Yin splitting algorithm is upper bounded by $2-\frac{\gamma}{2\beta}$ and that its value has an important effect. If $\gamma\in{]0,2\beta[}$, \emph{overrelaxed steps} (i.e., $\lambda_k>1$) are allowed in~\eqref{eq:DY0}, while only \emph{underrelaxed steps} can be taken when $\gamma\geq 2\beta$. The fact that both the stepsize and the relaxation parameters are important is especially apparent when one considers the particular case of $A=B=0$ and $T=\nabla f$ for a differentiable function $f$ whose gradient is Lipschitz continuous with constant $L=\frac{1}{\beta}$. In this case, the iteration~\eqref{eq:DY0} reduces to the gradient descent scheme:
\begin{equation}\label{eq:gradient_descent}
x_{k+1}=x_k-\gamma\lambda_k\nabla f(x_k),\quad k=0,1,2,\ldots.
\end{equation}
We observe in~\eqref{eq:gradient_descent} that the stepsize of the algorithm is actually $\gamma\lambda_k$, so the upper bound $2-\frac{\gamma}{2\beta}$ on the relaxation parameters $\lambda_k$ entails $\gamma\lambda_k<2\beta=\frac{2}{L}$, as expected.

Finally, it is important to recall that in practical applications only a lower bound of the best cocoercivity constant~$\beta$ is usually known, and this can affect the performance of the algorithms. For instance, consider again the application of the Davis--Yin algorithm with starting point $\widetilde{x}_0$ shown on the right in Figure~\ref{fig:2} and imagine that we underestimate $\beta$ to $\widehat{\beta}=0.65<1=\beta$. Then, we observe in Figure~\ref{fig:6} how the choice of a stepsize parameter $\gamma\in{]0,2\widehat{\beta}[}$ excludes better values like $\widehat{\gamma}\in{]2\widehat{\beta},4\widehat{\beta}[}$.
A typical choice for the parameters of the forward-backward algorithm is $\gamma=(2-\varepsilon)\beta$ and $\lambda_k=1$, for a small $\varepsilon>0$ (see, e.g.,~\cite{Condat2020}).
This example shows that, when only an estimate $\widehat{\beta}$ of the best value of $\beta$ is known, it can be worth testing the performance of the algorithm
with parameters $\gamma=(2+\varepsilon)\widehat{\beta}$ and $\lambda_k=1-\varepsilon$ (i.e., with underrelaxation).
%This choice would not be allowed by the current results in the literature requiring $\gamma<2\beta$ (see, e.g., \cite[Theorem~25.8]{bauschke2017}), but it is permitted by Theorem~\ref{t:GradientDR} as long as $\varepsilon<2$.

\begin{figure}[ht!]\centering
\includegraphics[width=.6\textwidth]{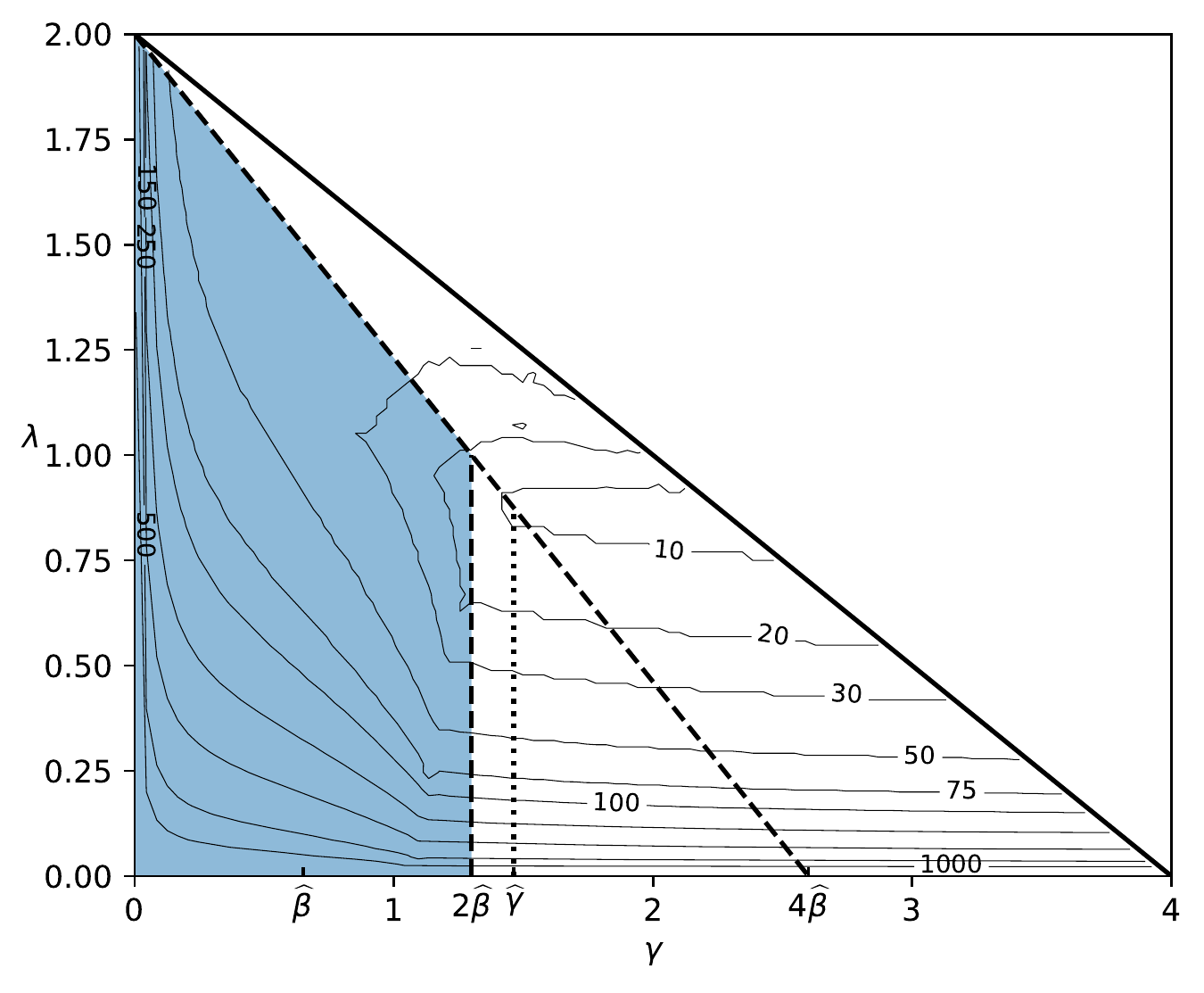}
\caption{Repetition of the experiment shown on the right of Figure~\ref{fig:2}. When only an approximate value $\widehat{\beta}$ of the cocoercivity constant is known, choosing the stepsize~$\gamma\in{]0,2\widehat{\beta}[}$ (shaded area) can exclude better choices like $\widehat{\gamma}$.}\label{fig:6}
\end{figure}

The remainder of this paper is structured as follows. In Section~\ref{s:preliminaries} we recall some preliminary notions and results. In Section~\ref{sec:3} we provide an alternative proof of convergence of the Davis--Yin splitting algorithm and derive its \emph{strengthened} version for computing the resolvent of the sum. In Section~\ref{sec:4} we include some illustrative numerical experiments. We finish with some conclusions in Section~\ref{sec:5}.

\section{Preliminaries}\label{s:preliminaries}
Throughout this paper, $\Hilbert$ is a real Hilbert space equipped with inner product $\langle\cdot , \cdot\rangle$ and induced norm $\|\cdot\|$. We abbreviate \emph{norm convergence} of sequences in $\Hilbert$ with $\to$  and we use~$\rightharpoonup$ for \emph{weak convergence}.

A \emph{set-valued operator} is a mapping $A:\Hi\setto\Hilbert$ that assigns to each point in $\Hi$ a subset of $\Hilbert$, i.e., $A(x)\subseteq \Hilbert$ for all $x\in \Hi$. In the case when $A$ always maps to singletons, i.e., $A(x)=\{u\}$ for all $x\in \Hi$, $A$ is said to be a \emph{single-valued mapping} and is denoted by $A:\Hi\to\Hilbert$. In an abuse of notation, we may write $A(x)=u$ when $A(x)=\{u\}$.  The \emph{domain}, the \emph{range}, the \emph{graph}, the set of \emph{fixed points} and the set of \emph{zeros} of $A$, are denoted, respectively, by $\dom A$, $\ran A$, $\gra A$, $\Fix A$ and $\zer A$;~i.e.,
\begin{gather*}
\begin{align*}
\dom A&:=\left\{x\in\Hilbert : A(x)\neq\emptyset\right\},&
\ran A&:=\left\{u\in\Hilbert : \exists x\in\Hilbert: u\in A(x) \right\},&\\
\gra A&:=\left\{(x,u)\in\Hilbert\times\Hilbert : u\in A(x)\right\},&
\Fix A&:=\left\{x\in\Hilbert : x\in A(x)\right\},
\end{align*}\\
\text{and} \quad \zer A:=\left\{x\in\Hilbert : 0\in A(x)\right\}.
\end{gather*}
The \emph{inverse operator} of $A$, denoted by $A^{-1}$, is defined through
$x\in A^{-1}(u) \iff u\in A(x)$. The \emph{identity operator} is denoted by $\Id$.

\begin{definition}\label{def:cocoercive}
We say that an operator $T:\Hi \to \Hilbert$ is
\begin{enumerate}[(i)]
\item\emph{$L$-Lipschitz continuous} for $L >0$ if
\begin{equation*}
\|T(x)-T(y)\| \leq L \|x-y\| \quad \forall x,y \in \Hi;
\end{equation*}
\item\emph{$\beta$-cocoercive} for $\beta >0$ if
\begin{equation*}
\langle x-y, T(x)-T(y) \rangle \geq \beta \|T(x)- T(y)\|^2 \quad \forall x,y \in \Hi.
\end{equation*}
\end{enumerate}
\end{definition}
Note that, by the  Cauchy--Schwarz inequality, any $\beta$-cocoercive mapping is $\frac{1}{\beta}$-Lipschitz continuous. When the operator is the gradient of a convex function, the Baillon--Haddad theorem states that both notions are equivalent, see \cite[Corolaire~10]{BaillonHaddad}.
\begin{definition}\label{def:monotonicity}
Let $A:\Hilbert\setto\Hilbert$ be a set-valued operator.
\begin{enumerate}[(i)]
\item A is said to be \emph{$\eta$-monotone} for $\eta\in\mathbb{R}$ if
$$ \langle x-y,u-v\rangle \geq \eta\|x-y\|^2\quad\forall(x,u),(y,v)\in\gra A. $$
Furthermore, an $\eta$-monotone operator $A$ is said to be \emph{maximally $\eta$-monotone} if there exists no $\eta$-monotone operator $B\colon\Hilbert\setto\Hilbert$ such that $\gra B$ properly contains $\gra A$.
\item $A$ is said to be \emph{uniformly monotone} with modulus $\phi : \mathbb{R}_+ \to [0, +\infty [$ if $\phi$ is increasing, vanishes only at $0$, and
$$ \langle x-y, u-v\rangle \geq \phi(\|x-y\|) \quad \forall (x,u),(y,v) \in \gra A.$$
\end{enumerate}
\end{definition}
An operator is monotone (in the classical sense) if it is $0$-monotone and it is $\eta$-strongly monotone (in the classical sense) if it is $\eta$-monotone for $\eta>0$, in which case it is uniformly monotone with modulus $\phi(t)=\eta t^2$, for $t\in\R_+$.

\begin{definition}\label{def:demiregularity }
We say that an operator $T:\Hilbert \to \Hilbert$ is \emph{demiregular} at $x \in \Hi$ if for all sequences $(x_k)_{k\in\mathbb{N}}$ with $x_k \wto x$ and $T(x_k) \to T(x)$, we have $x_k \to x$.
\end{definition}

The resolvent operator, whose definition is given next, is one of the main building blocks of splitting algorithms.

\begin{definition}\label{def:resolvent}
Given an operator $A\colon\Hilbert\setto\Hilbert$, the \emph{resolvent} of $A$ with parameter $\gamma>0$ is the operator $J_{\gamma A}\colon\Hilbert\setto\Hilbert$ defined by $J_{\gamma A}:=(\Id+\gamma A)^{-1}$.
\end{definition}

The following result is a consequence of Minty's theorem~\cite{Minty}.

\begin{proposition}[Resolvents of $\eta$-monotone operators]\label{prop:Jalpha}
Let $A:\Hilbert\setto\Hilbert$ be $\eta$-monotone and let $\gamma>0$ such that $1+\gamma\eta>0$. Then
\begin{enumerate}[(i)]
\item $J_{\gamma A}$ is single-valued,% and $(1+\gamma\alpha)$-cocoercive; (Lemma 3.3)
\item $\dom J_{\gamma A}=\Hilbert$ if and only if $A$ is maximally $\eta$-monotone.
\end{enumerate}
\end{proposition}
\begin{proof}
See~\cite[Proposition 3.4]{dao2019adaptive}.
\end{proof}
\begin{example}\label{ex:prox}
Let $f: \Hilbert \to ]-\infty, +\infty]$ be a proper, lower semicontinuous (lsc)
and convex function. Then, the \emph{subdifferential} of $f$, which is the operator $\partial f:\Hilbert \setto \Hilbert$ defined as
$$ \partial f(x) = \{u \in \Hilbert : f(x) + \langle u, y-x \rangle \leq f(y), \; \forall y \in \Hilbert\},$$
is a maximally monotone operator. Furthermore, it holds that $J_{\gamma \partial f} = \prox_{\gamma f}: \Hilbert \setto \Hilbert$, where $\prox_{\gamma f}$ is the \emph{proximity operator} of $f$ (with parameter $\gamma$) defined at $x \in \Hilbert$ by
$$ \prox_{\gamma f}(x) := \argmin_{u \in \Hilbert} \left( f(u)+ \frac{1}{2\gamma}\|x-u\|^2\right),$$
see, e.g.,~\cite[Theorem~20.25 \& Example~23.3]{bauschke2017}. Some functions are \emph{prox-friendly}, which means that their proximity operator is easy to compute, see~\cite{CombettesPesquet} for various examples. This is the case for the $\ell_1$-norm, whose proximity operator is the result of applying the \emph{soft thresholding} function:
$$\prox_{\gamma \|\cdot\|_1}(x)=\sign(x)\odot[|x|-\gamma]_{+},$$
where $\odot$ denotes element-wise product and $[\,\cdot\,]_{+}$ and $|\cdot|$ are applied element-wise. That is, its $i$-th component is given by
$$
\prox_{\gamma \|\cdot\|_1}(x)_i
=
\left\{ \begin{array}{ll}
x_i + \gamma, & \text{if} \; x_i <  -\gamma , \\
0, & \text{if} \; |x_i| \leq \gamma, \\
x_i - \gamma, &  \text{if} \; x_i > \gamma,
\end{array}
\right.
$$
for  $i=1,2,\ldots,n$.
\end{example}
\begin{example}\label{ex:2}
Given a nonempty set $C \subseteq \Hilbert$, the  \emph{indicator function} of $C$, $\iota_C: \Hilbert \to {]-\infty, \infty]}$, is defined as
$$ \iota_C(x):=
\left\{ \begin{array}{ll}
0, & \text{if} \; x \in C; \\
+\infty, & \text{if} \; x \notin C.
\end{array}
\right. $$
When $C$ is a convex set, $\iota_C$ is a convex function whose subdifferential becomes the \emph{normal cone} to $C$, $N_C: \Hilbert \to \Hilbert$, given by
$$ \partial \iota_C(x) = N_C(x) :=
\left\{\begin{array}{ll}
\{u \in \Hilbert: \langle u, c-x\rangle \leq 0, \; \forall c \in C \}, & \text{if} \; x \in C, \\
\emptyset, & \text{otherwise}.
\end{array}
\right.
$$
When $C$ is nonempty, closed and convex, the normal cone $N_C$ is maximally monotone. Furthermore, $J_{N_C} = P_C$, where $P_C:\Hilbert \to \Hilbert$ denotes the \emph{projector} onto $C$, which is defined at $x\in\Hilbert$ by
$$ P_C(x) := \argmin_{c \in C} \|x-c\|, $$
see, e.g.,~\cite[Example~20.26 \& Example~23.4]{bauschke2017}.
\end{example}
Fej\'er monotonicity is a key property in fixed point theory (see, e.g,~\cite[Chapter~5]{bauschke2017}). It will allow us to derive weak convergence of the sequence generated by the Davis--Yin splitting algorithm.
\begin{definition}\label{def:Fejer}
Let $C$ be a nonempty subset of $\Hilbert$ and let $(x_n)_{n\in\mathbb{N}}$ be a sequence in $\Hilbert$. Then $(x_n)_{n\in\mathbb{N}}$ is \emph{Fej\'er monotone} with respect to $C$ if for all $x \in C$
$$ \|x_{n+1}-x\| \leq \|x_n - x\| \quad \forall n \in \mathbb{N}. $$
\end{definition}
\begin{proposition}\label{p:Fejer}
Let $C$ be a nonempty subset of $\Hilbert$ and let $(x_n)_{n\in\mathbb{N}}$ be a sequence in $\Hilbert$. Suppose that $(x_n)_{n\in\mathbb{N}}$ is Fej\'er monotone with respect to $C$ and that every weak sequential cluster point of $(x_n)_{n\in\mathbb{N}}$ belongs to $C$. Then $(x_n)_{n\in\mathbb{N}}$ converges weakly to a point in $C$.
\end{proposition}
\begin{proof}
See~\cite[Theorem 5.5]{bauschke2017}.
\end{proof}

\section{Davis--Yin splitting algorithm}\label{sec:3}

Let $A, B : \Hilbert \setto \Hilbert$ be two maximally monotone operators and let $T:\Hilbert \to \Hilbert$ be cocoercive. Consider the problem
\begin{equation}\label{eq:Problem3Op}
\text{find~}x\in \Hilbert \text{ such that } 0 \in (A+B+T)(x).
\end{equation}
The following lemma characterizes the set of zeros of the latter sum of operators in terms of the set
\begin{equation}\label{eq:Omega}
\Omega_\gamma := \bigl\{ x \in \Hilbert: J_{\gamma A}(x) = J_{\gamma B}\bigl(2 J_{\gamma A}(x)-x-\gamma {T} (J_{\gamma A}(x))\bigr) \bigr\},
\end{equation}
with $\gamma >0$, and shows that $\Omega_\gamma=\Fix DY_\gamma$, where
\begin{equation}
\label{e:FixDY}
\Fix DY_\gamma=\left\{u+\gamma y: u\in\zer (A+B+T), y\in \left(-B(u)-T(u)\right) \cap A(u)\right\},
\end{equation}
as shown in~\cite[Lemma~2.2]{davis2017three}.
\begin{lemma}\label{l:solP3Op}
For every $\gamma>0$, it holds
$$\zer(A+B+T) = J_{\gamma A} (\Omega_\gamma).$$
In particular, $\zer \left(A+B+T\right)\neq\emptyset\iff\Omega_\gamma\neq\emptyset$. Further, $\Omega_\gamma=\Fix DY_\gamma$.
\end{lemma}

\begin{proof}
Observe that
\begin{equation*}
\begin{aligned}
u  \in \zer\left(A+B+T\right) & \Leftrightarrow -\gamma{T}(u) \in (\gamma A+ \gamma B)(u) \\
& \Leftrightarrow (\exists \, x \in \Hilbert) \quad x-u \in \gamma A(u),  \quad u-x-\gamma{T}(u) \in \gamma B(u) \\
& \Leftrightarrow (\exists \, x \in \Hilbert) \quad u = J_{\gamma A}(x), \quad 2u-x-\gamma {T}(u) \in (\Id + \gamma B)(u) \\
& \Leftrightarrow (\exists \, x \in \Hilbert) \quad u  = J_{\gamma A}(x), \quad u = J_{\gamma B}(2u-x-\gamma {T}(u)),
\end{aligned}
\end{equation*}
from where the first claim follows. Further, we have
\begin{equation*}
\begin{aligned}
x \in \Omega_\gamma & \Leftrightarrow (\exists u \in \zer \left(A+B+T\right)) \quad  u = J_{\gamma A}(x), \quad u = J_{\gamma B}(2u - x - \gamma T(u) ) \\
& \Leftrightarrow  (\exists u \in \zer \left(A+B+T\right)) \quad x-u \in \gamma A(u), \quad x-u \in (-\gamma B(u)-\gamma T(u)) \\
& \Leftrightarrow (\exists u \in \zer \left(A+B+T\right),\;\exists y \in  \left(-B(u) - T(u)\right) \cap A(u)), \quad x = u + \gamma y,
\end{aligned}
\end{equation*}
and thus, $\Omega_\gamma=\Fix DY_\gamma$, by~\eqref{e:FixDY}.
\end{proof}

Using a technique similar to the one employed in~\cite[Theorem~8]{franstrengthening}, we can provide a direct proof of the convergence of Davis--Yin splitting algorithm with the additional advantages of both allowing a larger stepsize and having a simpler condition on the relaxation parameters than~\cite[Theorem~2.1]{davis2017three}. The proof makes use of the following technical lemma.

\begin{lemma}\label{lemma2}
Let $A,B: \Hilbert \setto \Hilbert$ be two maximally monotone operators and ${T}:\Hilbert \to \Hilbert$. Let $x,\hat{x}\in\Hilbert$ and $\gamma>0$, and set $u:=J_{\gamma A}(x)$, $\hat{u}:=J_{\gamma_A}(\hat{x})$, $v:=J_{\gamma B}(2u-x-\gamma T(u))$ and $\hat{v}:=J_{\gamma B}(2\hat{u}-\hat{x}-\gamma T(\hat{u}))$. Then, it holds
\begin{equation}\label{e:lemma2_0}
0\leq \langle x-\hat{x},(u-v)-(\hat{u}-\hat{v})\rangle-\|(u-v)-(\hat{u}-\hat{v})\|^2-\gamma\langle T(u)-T(\hat{u}),v-\hat{v}\rangle.
\end{equation}
Further, if $A$ (respectively $B$) is uniformly monotone with modulus $\phi$, then~\eqref{e:lemma2_0} holds with $0$ replaced by $\gamma\phi(\|u-\hat{u}\|)$ (respectively $\gamma\phi(\|v-\hat{v}\|)$).
\end{lemma}

\begin{proof}
Since $x-u\in\gamma A(u)$ and $\hat{x}-\hat{u}\in\gamma A(\hat{u})$, monotonicity of $\gamma A$ yields
\begin{equation}\label{e:lemma2_1}
0\leq \langle (x-u)-(\hat{x}-\hat{u}),u-\hat{u}\rangle.
\end{equation}
Likewise, since $2u-x-\gamma T(u)-v\in\gamma B(v)$ and $2\hat{u}-\hat{x}-\gamma T(\hat{u})-\hat{v}\in\gamma B(\hat{v})$, monotonicity of $\gamma B$ implies
\begin{equation}\label{e:lemma2_2}
\begin{aligned}
0&\leq \langle (2u-x-\gamma T(u)-v)-(2\hat{u}-\hat{x}-\gamma T(\hat{u})-\hat{v}),v-\hat{v}\rangle\\
&=\langle (\hat{v}-\hat{u})-(v-u),v-\hat{v}\rangle-\langle(x-u)-(\hat{x}-\hat{u}),v-\hat{v}\rangle -\gamma\langle T(u)-T(\hat{u}),v-\hat{v}\rangle.
\end{aligned}
\end{equation}
Summing together~\eqref{e:lemma2_1} and~\eqref{e:lemma2_2}, we obtain
\begin{align*}
0&\leq\langle(x-u)-(\hat{x}-\hat{u}),(u-v)-(\hat{u}-\hat{v})\rangle+\langle(\hat{v}-\hat{u})-(v-u),v-\hat{v}\rangle-\gamma\langle T(u)-T(\hat{u}),v-\hat{v}\rangle\\
&=\langle x-\hat{x},(u-v)-(\hat{u}-\hat{v})\rangle-\|(u-v)-(\hat{u}-\hat{v})\|^2-\gamma\langle T(u)-T(\hat{u}),v-\hat{v}\rangle,
\end{align*}
which proves~\eqref{e:lemma2_0}. The last assertion easily follows from the definition of uniform monotonicity.
\end{proof}

\begin{theorem}[{Davis--Yin splitting}]\label{t:GradientDR}
Let $A,B: \Hilbert \setto \Hilbert$ be two maximally monotone operators and ${T}:\Hilbert \to \Hilbert$ be a $\beta$-cocoercive operator, with $\beta >0$, such that $\zer\left(A+B+T\right)  \neq \emptyset$.  Set a stepsize $\gamma\in{]0,4\beta[}$ and consider a sequence of relaxation parameters $(\lambda_k)_{k\in\mathbb{N}}$ in $]0,2-\gamma/(2\beta)]$ such that $\sum_{k\in\mathbb{N}}\lambda_k\left(2-\frac{\gamma}{2\beta}-\lambda_k\right)=+\infty$. Given some initial point $x_0 \in \Hilbert$, consider the sequences defined by
\begin{equation}\label{eq:GradientDR}
\left\{\begin{aligned}
  u_k &= J_{\gamma A}(x_k) \\
  v_k &= J_{\gamma B}(2 u_k -x_k - \gamma {T}(u_k)) \\
  x_{k+1} &= x_k + \lambda_k (v_k-u_k).
\end{aligned}\right.
\end{equation}
Then, the sequence $(x_k)_{k\in\mathbb{N}}$ is Fej\'er monotone with respect to the set $\Omega_\gamma$ given in~\eqref{eq:Omega}.
Moreover, the following assertions hold:
\begin{enumerate}[(i)]
\item\label{it:GradientDR-ii}  $x_k \wto \bar{x} \in \Omega_\gamma$, $u_k \wto \bar{u}$, $v_k \wto \bar{u}$, $v_k-u_k\to 0$ and $T(u_k)\to T(\bar{u})$ with
\begin{equation}\label{eq:wconvergence}
\bar{u} = J_{\gamma A}(\bar{x}) = J_{\gamma B}(2 \bar{u}- \bar{x} - \gamma {T} (\bar{u})) \in \zer \left(A+B+T\right).
\end{equation}%
Further, $T\left(\zer(A+B+T)\right)=\{T(\bar{u})\}$.
\item\label{it:GradientDR-iii} If either $A$ or $B$ is uniformly monotone on every bounded subset of its domain, or $T$ is demiregular at every point  in $\zer\left(A+B+T\right)$, then  $(u_k)_{k\in\mathbb{N}}$ and $(v_k)_{k\in\mathbb{N}}$ converge strongly to $\bar{u} \in \zer\left(A+B+T\right)$.
\end{enumerate}
\end{theorem}

\begin{proof}
Define the sequences
\begin{equation*}
(\forall k \in \mathbb{N})  \qquad z_k := \gamma {T}(u_k) \qquad \text{and} \qquad w_k := v_k-u_k
\end{equation*}
and note the following relations that \eqref{eq:GradientDR} yields
\begin{equation}\label{eq:graphrelations}
(u_k, x_k-u_k)  \in \gra{\gamma A}\quad\text{and}\quad (v_k, 2u_k-x_k-z_k-v_k) \in \gra{\gamma B}.
%\left\{\begin{aligned}
%(u_k, y_k) & \in \gra{\gamma A} \\
%(v_k, w_k) & \in \gra{\gamma B} \\
%y_k +w_k +z_k & = u_k - v_k.
%\end{aligned}\right.
\end{equation}

Pick any $x \in \Omega_\gamma$ and denote $u:=J_{\gamma A}(x)$. By definition of $\Omega_\gamma$, we have $u= J_{\gamma B}(2u-x-\gamma {T}(u))$. Applying Lemma~\ref{lemma2} to $x$ and $\hat{x}:= x_k$, observing that $\hat{u}=u_k$, $v=u$ and $\hat{v}=v_k$, yields
\begin{align}
0 & \leq \langle x-x_k, w_k \rangle - \|w_k\|^2 - \gamma \langle {T}(u)-{T}(u_k), u-v_k \rangle.\label{eq:monoAB}
\end{align}
The first two terms in \eqref{eq:monoAB} multiplied by $2\lambda_k$ can be expressed as
\begin{equation}\label{eq:first2terms}
\begin{aligned}
2\lambda_k \left( \langle x-x_k, w_k \rangle - \|w_k\|^2 \right)  &= 2 \langle x-x_k, x_{k+1}-x_k \rangle - 2\lambda_k \|w_k\|^2 \\
& = \|x_k-x\|^2-\|x_{k+1}-x\|^2+\lambda_k(\lambda_k-2) \|w_k\|^2.
\end{aligned}
\end{equation}
Now, using the $\beta$-cocoercivity of ${T}$, the last term in \eqref{eq:monoAB} can be expressed as
\begin{equation}\label{eq:lastterm}
\begin{aligned}
- \gamma \langle {T}(u)-{T}(u_k), u-v_k \rangle
& = -\gamma \langle {T}(u)-{T}(u_k), u-u_k \rangle + \gamma \langle {T}(u)-{T}(u_k), w_k \rangle \\
& \leq - \beta\gamma \|T(u)-T(u_k)\|^2 + \gamma \langle {T}(u)-{T}(u_k), w_k \rangle  .
\end{aligned}
\end{equation}
Using Cauchy--Schwarz and Young's inequalities,  the last term in \eqref{eq:lastterm} can be estimated as
\begin{equation}\label{eq:lastterm1}
\begin{aligned}
\gamma\langle {T}(u)-{T}(u_k), w_k \rangle & \leq  \beta\gamma \| {T}(u)-{T}(u_k) \|^2+ \frac{\gamma}{4\beta} \|w_k\|^2.
\end{aligned}
\end{equation}
Combining \eqref{eq:monoAB}-\eqref{eq:lastterm1}, we have
\begin{equation*}
\|x_{k+1}-x\|^2 + \lambda_k(2-\lambda_k)\|w_k\|^2 \leq \|x_k-x\|^2  + \frac{2\gamma\lambda_k}{4\beta} \|w_k\|^2.
\end{equation*}
As a result, we reach the expression
\begin{equation}\label{eq:Fmon}
\|x_{k+1}-x\|^2 + \lambda_k\left(2-\frac{\gamma}{2\beta}-\lambda_k\right) \|w_k\|^2 \leq \|x_{k}-x\|^2.
\end{equation}
Since $\lambda_k\leq 2-\gamma/(2\beta)$, equation \eqref{eq:Fmon} implies that $(x_k)_{k\in\mathbb{N}}$ is Fej\'er monotone with respect to~$\Omega_\gamma$ and thus, bounded.  Since resolvents are nonexpansive and ${T}$ is $\frac{1}{\beta}$-Lipschitz continuous (by Cauchy--Schwarz), it follows that $(u_k)_{k\in\mathbb{N}}$, $(z_k)_{k\in\mathbb{N}}$ and $(v_k)_{k\in\mathbb{N}}$ are bounded.

\eqref{it:GradientDR-ii}:~ The Fej\'er monotonicity of $(x_k)_{k\in\mathbb{N}}$ implies that the sequence $(\|x_k-x\|)_{k\in\mathbb{N}}$ is nonincreasing and convergent. Telescoping~\eqref{eq:Fmon}, we obtain
$$\sum_{k\in\mathbb{N}}\lambda_k\left(2-\frac{\gamma}{2\beta}-\lambda_k\right)\|w_k\|^2\leq \|x_0-x\|^2,$$
which implies $\liminf_{k\to\infty}\|w_k\| = 0$, since $\sum_{k\in\mathbb{N}}\lambda_k\left(2-\frac{\gamma}{2\beta}-\lambda_k\right)=+\infty$. 
To prove that $w_k\to 0$, it suffices to show that the sequence $(\|w_k\|)_{k\in\mathbb{N}}$ is nonincreasing. Applying Lemma~\ref{lemma2} with $x:=x_{k+1}$ and $\hat{x}:=x_k$ yields
%\begin{align*}
$$0\leq \langle x_{k+1}-x_k,w_k-w_{k+1}\rangle-\|w_{k+1}-w_k\|^2-\gamma\langle T(u_{k+1})-T(u_k),v_{k+1}-v_k\rangle.$$
%&\leq \langle \lambda_k w_k,w_{k+1}-w_k\rangle-\|w_{k+1}-w_k\|^2-\gamma\langle T(u_{k+1})-T(u_k),v_{k+1}-v_k\rangle.
%\end{align*}
The first two terms multiplied by $2$ can be expressed as
$$2\langle \lambda_k w_k,w_k-w_{k+1}\rangle-2\|w_{k+1}-w_k\|^2=\lambda_k^2\|w_k\|^2-\|w_{k+1}-w_k\|^2-\|w_{k+1}-w_k+\lambda_kw_k\|^2,$$
while the third term is equal to
\begin{align*}
-\gamma\langle T(u_{k+1})&-T(u_k),v_{k+1}-v_k\rangle\\
&=-\gamma\langle T(u_{k+1})-T(u_k),w_{k+1}-w_k\rangle-\gamma\langle T(u_{k+1})-T(u_k),u_{k+1}-u_k\rangle\\
&\leq \gamma\beta\|T(u_{k+1})-T(u_k)\|^2+\frac{\gamma}{4\beta}\|w_{k+1}-w_k\|^2-\gamma\beta\|T(u_{k+1})-T(u_k)\|^2\\
&=\frac{\gamma}{4\beta}\|w_{k+1}-w_k\|^2,
\end{align*}
where we have used again Young's inequality and the cocoercivity of $T$. Therefore, we deduce
\begin{align*}
0&\leq \lambda_k^2\|w_k\|^2-\left\|w_{k+1}-w_k+\lambda_kw_k\right\|^2+\left(\frac{\gamma}{2\beta}-1\right)\|w_{k+1}-w_k\|^2\\
&=\lambda_k^2\|w_k\|^2-\lambda_k^2\|w_k\|^2+2\lambda_k\langle w_{k+1}-w_k,-w_k\rangle+\left(\frac{\gamma}{2\beta}-2\right)\|w_{k+1}-w_k\|^2\\
&=\lambda_k\|w_k\|^2-\lambda_k\|w_{k+1}\|^2+\left(\frac{\gamma}{2\beta}-2+\lambda_k\right)\|w_{k+1}-w_k\|^2,
\end{align*}
that is,
$$\lambda_k\|w_{k+1}\|^2\leq \lambda_k\|w_k\|^2-\left(2-\frac{\gamma}{2\beta}-\lambda_k\right)\|w_{k+1}-w_k\|^2\leq\lambda_k\|w_k\|^2,$$
so $(\|w_k\|)_{k\in\mathbb{N}}$ is nonincreasing, since $\lambda_k>0$. Hence, we have proved that $w_k\to 0$.

Let $(\bar{x},\bar{u},\bar{z})$ be a weak sequential cluster point of the bounded sequence $(x_k,u_k,z_k)_{k\in\mathbb{N}}$. Hence, there is a subsequence of $(x_{k_n},u_{k_n},z_{k_n})_{n\in\mathbb{N}}$ which is weakly convergent to $(\bar{x},\bar{u},\bar{z})$. Now, consider the operator $S: \Hilbert^3 \setto \Hilbert^3$ given by
$$S:=
\colvec{(\gamma A)^{-1}\\(\gamma{T})^{-1}\\ \gamma B} +
\colvec{0&0&-\Id\\0&0&-\Id\\\Id&\Id&0},
$$
which is maximally monotone, because it is the sum of a maximally monotone operator and a skew-symmetric matrix (see, e.g., \cite[Example~20.35~\&~Corollary~25.5(i)]{bauschke2017}).
From~\eqref{eq:graphrelations}, it follows that
$$
\colvec{u_{k_n}-v_{k_n}\\u_{k_n}-v_{k_n}\\u_{k_n}-v_{k_n}} \in S \colvec{x_{k_n}-u_{k_n}\\z_{k_n}\\v_{k_n}}.
$$
As the graph of a maximally monotone operator is sequentially closed in the weak-strong topology (see, e.g., \cite[Proposition~20.38]{bauschke2017}), taking the limit as $n\to\infty$ and observing that $x_{k_n} -u_{k_n}\wto \bar{x}-\bar{u}$ and $v_{k_n} \wto \bar{u}$ (since $w_{k_n}=v_{k_n}-u_{k_n} \to 0$), we deduce that
$$\colvec{0 \\0\\0} \in
\left(
\colvec{\left(\gamma A\right)^{-1} \\ \left(\gamma {T}\right)^{-1} \\ \gamma B} +
\colvec{0&0&-\Id \\ 0&0&-\Id \\ \Id&\Id&0}
\right)
\colvec{\bar{x}-\bar{u} \\ \bar{z} \\ \bar{u}}.$$
The latter inclusion is equivalent to
\begin{equation}\label{eq:ubarOmega}
\bar{u} = J_{\gamma A}(\bar{x}),\quad \bar{z} = \gamma {T}(\bar{u})\quad\text{and}\quad
\bar{u} = J_{\gamma B} (2 \bar{u}-\bar{x}- \bar{z}),
\end{equation}
which implies $\bar{x}\in\Omega_\gamma$. Therefore, every weak sequential cluster point of $(x_k)_{k\in\mathbb{N}}$ is contained in~$\Omega_\gamma$, and Proposition~\ref{p:Fejer} implies that  $(x_k)_{k\in\mathbb{N}}$ is weakly convergent to a point $\bar{x} \in \Omega_\gamma$. Then~\eqref{eq:ubarOmega} shows that $\bar{u}= J_{\gamma A}(\bar{x})$ and $\bar{z} = \gamma T(\bar{u})$ are the unique cluster points of $(u_k)_{k\in\mathbb{N}}$ and $(z_k)_{k\in\mathbb{N}}$, respectively, and hence $u_k \wto \bar{u}$, $v_k \wto \bar{u}$ and $z_k \wto \bar{z}$.

Moreover, since $x$ was arbitrarily chosen in $\Omega_\gamma$,~\eqref{eq:monoAB} and~\eqref{eq:lastterm} also hold with $u$ replaced by $\bar{u}$ and $x$ replaced by~$\bar{x}$. From the resulting inequalities, we obtain
\begin{equation}\label{e:strongT(u)}
\begin{aligned}
\beta \gamma \|T(\bar{u})-T(u_k)\|^2 \leq\, &  \langle \bar{x}-x_k,w_k\rangle + \langle u_k-\bar{u},w_k \rangle\\ &+ \langle \bar{u}-v_k, w_k \rangle + \gamma \langle T(\bar{u})-T(u_k), w_k \rangle,
\end{aligned}
\end{equation}
and thus $T(u_k) \to T(\bar{u})$.
Now, by Lemma~\ref{l:solP3Op},  we know that $\bar{u} \in \zer \left(A+ B+T\right)$.

Finally, pick any $\tilde{u} \in \zer\left(A+B+T\right)$. By Lemma~\ref{l:solP3Op}, there is $\tilde{x} \in \Omega_{\gamma}$ such that $\tilde{u} = J_{\gamma A}(\tilde{x})$. Setting $x=\tilde{x}$ at the beginning of the proof,~\eqref{e:strongT(u)} becomes
\begin{align*}
%\label{e:strongT(u)2}
\beta \gamma \|T(\tilde{u})-T(u_k)\|^2 \leq\, & \langle \tilde{x}-x_k,w_k\rangle + \langle u_k-\tilde{u},w_k \rangle+ \langle \tilde{u}-v_k, w_k \rangle + \gamma \langle T(\tilde{u})-T(u_k), w_k \rangle.
\end{align*}
Since $x_k \wto \bar{x}$, $u_k \wto \bar{u}$, $v_k \wto \bar{u}$, $w_k \to 0$ and $T(u_k) \to T(\bar{u})$, the inequality above implies $T(\bar{u}) = T(\tilde{u})$. This proves that $T\left(\zer(A+B+T)\right)=\{T(\bar{u})\}$.

\eqref{it:GradientDR-iii}:~Assume first that $A$ is uniformly monotone. Since the sequence $(u_k)_{k\in\mathbb{N}}$ is bounded, the set $\{\bar{u}\}\cup\{u_k, k\geq 0\}\subset \dom A$ is bounded. Thus, using uniform monotonicity in Lemma~\ref{lemma2} with $x:=\bar{x}$ and $\hat{x}:=x_k$, we obtain the stronger inequality
\begin{equation}
\gamma\phi(\|\bar{u}-u_k\|)\leq \langle \bar{x}-x_k, w_k \rangle - \|w_k\|^2 - \gamma \langle {T}(\bar{u})-{T}(u_k), \bar{u}-v_k \rangle,
\end{equation}
which entails $\gamma\phi(\|\bar{u}-u_k\|)\to 0$. Since $\phi$ is increasing, we deduce that $u_k\to \bar{u}$, which implies $v_k\to \bar{u}$. When $B$ is uniformly monotone, the result similarly follows.

Finally, suppose that the demiregularity assumption holds. By (i), we know that $u_k\wto \bar{u}$ and $T(u_k)\to T(\bar{u})$, so the demiregularity of $T$ at $\bar{u}$ implies that $u_k\to \bar{u}$. Since $v_k-u_k\to 0$, we also obtain that $v_k\to \bar{u}$.
\end{proof}

\begin{remark}\label{rem:1} (i) The stepsize $\gamma$ in~\cite[Theorem~2.1]{davis2017three}  is assumed to be in ${]0,2\beta\varepsilon[}$, with $\varepsilon\in{]0,1[}$, while Theorem~\ref{t:GradientDR} allows to take stepsizes in the interval ${]0,4\beta[}$, which is twice larger. Note that our assumption is required to guarantee that  $2-\gamma/(2\beta)>0$. The relaxation parameters $(\lambda_k)_{k\in\mathbb{N}}$ in~\cite[Theorem~2.1]{davis2017three} must be taken in $]0,2-\varepsilon[$, while the interval given in Theorem~\ref{t:GradientDR} is $]0,2-\gamma/(2\beta)]$. If $\gamma\in{]0,2\beta\varepsilon[}$, we have $2-\varepsilon<2-\gamma/(2\beta)$. Thus, Theorem~\ref{t:GradientDR} additionally allows to take some of the relaxation parameters equal to $2-\gamma/(2\beta)$ (but not all of them, as we need $\sum_{k\in\mathbb{N}}\lambda_k\left(2-\frac{\gamma}{2\beta}-\lambda_k\right)=+\infty$, unless either $A$ or $B$ is uniformly monotone). Finally, unlike~\cite[Theorem~2.1]{davis2017three}, we do not require the assumption $\inf_{k\in\mathbb{N}}\lambda_k>0$.\\
(ii) In~Theorem~\ref{t:GradientDR}\eqref{it:GradientDR-iii}, even when $\sum_{k\in\mathbb{N}}\lambda_k\left(2-\frac{\gamma}{2\beta}-\lambda_k\right)<+\infty$, we have proved that the sequence $(u_k)_{k\in\mathbb{N}}$ (respectively $(v_k)_{k\in\mathbb{N}}$) is strongly convergent to $\bar{u}$ when $A$ (respectively $B$) is uniformly monotone.\\
(iii) Observe that it is also possible to prove $x_k\wto \bar{x}\in\Omega_\gamma$ using the notion of \emph{conically averaged operators} recently introduced in~\cite{BartzDaoPhan20}, not only for a fixed relaxation parameter $\lambda_k=\lambda$, as it was done in~\cite[Corollary~4.2]{DaoPhan21}. Indeed, by~\cite[Theorem~4.1]{DaoPhan21}, the operator $DY_\gamma$ in~\eqref{eq:DYoperator} is conically $(2-\gamma/(2\beta))^{-1}$-averaged, so~\cite[Proposition~2.9]{BartzDaoPhan20} can be applied to deduce the convergence of the Krasnosel'ski\u{i}--Mann iteration~\eqref{eq:DY0} to a fixed point of $DY_\gamma$, which belongs to $\Omega_\gamma$ by Lemma~\ref{l:solP3Op}.
\end{remark}

As a corollary, we obtain the following convergence result for the forward-backward splitting algorithm that allows doubling the range of the stepsizes assumed in~\cite[Theorem~26.14]{bauschke2017} (which is a particular case of~\cite[Proposition~4.4]{CombettesYamada}). Although this wider range of the stepsizes has been shown before in~\cite{patrinos17,Giselsson,Walaa}, it has not yet become widely known in the literature.

\begin{corollary}
\label{c:FB}
Let $B:\Hilbert \setto \Hilbert$ be a maximally monotone operator and $T:\Hilbert \to \Hilbert$ be a $\beta$-cocoercive operator, with $\beta >0$, such that $\zer \left(B+T\right) \neq \emptyset$. Set a stepsize $\gamma \in{]0,4\beta[}$ and consider a sequence of relaxation parameters $(\lambda_k)_{k\in\mathbb{N}}$  in $]0,2-\gamma/(2\beta)]$ such that $\sum_{k\in\mathbb{N}}\lambda_k\left(2-\frac{\gamma}{2\beta}-\lambda_k\right)= +\infty$. Given some initial point $x_0\in\Hilbert$, consider the sequences defined by
\begin{equation*}\label{e:FB}
\left\{\begin{aligned}
  y_k &= x_k-\gamma {T}(x_k) \\
  x_{k+1} &= x_k + \lambda_k (J_{\gamma B}(y_k)-x_k).
\end{aligned}\right.
\end{equation*}
Then, the following assertions hold:
\begin{enumerate}[(i)]
\item $(x_k)_{k\in\mathbb{N}}$ converges weakly to a point $\bar{x} \in \zer \left(B+T\right)$ and $\left(T(x_k)\right)_{k\in\mathbb{N}}$ converges strongly to the unique dual solution $T(\bar{x})$.
\item If either $B$ is uniformly monotone on every bounded subset of its domain, or T is demiregular at every point in $\zer \left(B+T\right)$, then $(x_k)_{k\in \mathbb{N}}$ converges strongly to $\bar{x} \in \zer\left(B+T\right)$.
\end{enumerate}
\end{corollary}
\begin{proof}
Apply Theorem~\ref{t:GradientDR} with $A=0$. By Theorem~\ref{t:GradientDR}\eqref{it:GradientDR-ii}, $T(x) \to T(\bar{x})$ and $T\left(\zer\left(B+T\right)\right)=\{T(\bar{x})\}$, which is the solution to the dual problem, see~\cite[Proposition~26.1(iv)]{bauschke2017}.
\end{proof}

%\section{Strengthened-Davis--Yin splitting algorithm}

We conclude this section by deriving a splitting algorithm for computing the resolvent of $A+B+T$. To this aim, we use the systematic framework developed in~\cite{franstrengthening}, based on the notion of strengthening of an operator.
\begin{definition}
Let $\theta>0$, $\sigma\in\R$ and let $w\in\Hi$. Given $A\colon\Hilbert\setto\Hilbert$, the \emph{$(\theta,\sigma)$-strengthening with inner perturbation $w$} of $A$ is the operator $A_w^{(\theta,\sigma)}\colon\Hilbert\setto\Hilbert$ defined by
\begin{equation*}\label{eq:strength def}
  A_w^{(\theta,\sigma)}:=A\circ\left(\theta\Id-w\right)+\sigma\Id.
\end{equation*}
\end{definition}

\begin{theorem}[Strengthened-Davis--Yin splitting]\label{t:DYsplitting} Let $A, B: \Hilbert \setto \Hilbert$ be maximally $\alpha_A$-monotone and $\alpha_B$-monotone operators, respectively, and let $T:\Hilbert \to \Hilbert$ be a $\beta$-cocoercive and maximally $\alpha_T$-monotone operator, with $\beta >0$. Let $\theta>0$, $\sigma_A,\sigma_B\in\R$ and $\sigma_T\geq 0$ be such that
\begin{equation}\label{eq:assumption}
\sigma_A+\sigma_B+\sigma_T>0\quad\text{and}\quad\left(\theta \alpha_A + \sigma_A,\theta\alpha_B+\sigma_B,\theta\alpha_T+\sigma_T\right)\in\R_{+}^3\setminus\{0_3\}.
\end{equation}
Let $\mu := (\theta / \beta + \sigma_T)^{-1}$ and $\gamma \in{]0,4\mu[}$. Consider a sequence of relaxation parameters $(\lambda_k)_{k \in \mathbb{N}}$ in $]0,2-\gamma/(2\mu)]$ verifying $\sum_{k\in \mathbb{N}}\lambda_k\left(2-\frac{\gamma}{2\mu}-\lambda_k\right) = +\infty$. Suppose $q \in \ran \left(\Id + \frac{\theta}{\sigma_A+\sigma_B+\sigma_T}\left(A+B+T\right)\right)$.  Given any $x_0 \in \Hilbert$, consider the sequences
\begin{equation}\label{eq:DYsplitting}
\left\{\begin{aligned}
u_k & = J_{\frac{\gamma \theta}{1+\gamma\sigma_A}A}\left(\frac{1}{1+\gamma \sigma_A}(x_k + \gamma \sigma_A q )\right) \\
v_k & = J_{\frac{\gamma\theta}{1+\gamma \sigma_B}B}\left(\frac{1}{1+\gamma\sigma_B}((2-\gamma \sigma_T)u_k-x_k-\theta\gamma T(u_k)+\gamma(\sigma_B+\sigma_T)q)\right) \\
x_{k+1} & = x_k + \lambda_k (v_k-u_k).
\end{aligned}\right.
\end{equation}
Then $(u_k)_{k\in\mathbb{N}}$ and $(v_k)_{k\in\mathbb{N}}$ are weakly convergent to $J_{\frac{\theta}{\sigma_A+\sigma_B+\sigma_T}(A+B+T)}(q)$, and  $(x_k)_{k\in\mathbb{N}}$ is weakly convergent to $\bar{x}$, with
$$ J_{\frac{\gamma \theta}{1+\gamma \sigma_A}A} \left(\frac{1}{1+\gamma\sigma_A}(\bar{x}+\gamma \sigma_A q)\right) = J_{\frac{\theta}{\sigma_A+\sigma_B+\sigma_T}(A+B+T)}(q).$$
Further,  if $\theta \alpha_A+\sigma_A >0$ (respectively $\theta_B \alpha + \sigma_B > 0$) then  the convergence of $(u_k)_{k\in\mathbb{N}}$  (respectively $(v_k)_{k\in\mathbb{N}}$) is strong, even when $\sum_{k\in \mathbb{N}} \lambda_k\left(2-\frac{\gamma}{2\mu}-\lambda_k\right) < +\infty$.
\end{theorem}
\begin{proof}
Set $\hat{x}_0  := \frac{1}{\theta}(x_0-q)$ and consider the sequences
\begin{equation}
\left\{\begin{aligned}\label{eq:DYstrengthened}
\hat{u}_k & = J_{\gamma A_{-q}^{(\theta,\sigma_A)}}(\hat{x}_k) \\
\hat{v}_k & = J_{\gamma B_{-q}^{(\theta,\sigma_B)}}\left(2\hat{u}_k-\hat{x}_k-\gamma T_{-q}^{(\theta, \sigma_T)}(\hat{u}_k)\right) \\
\hat{x}_{k+1} & = \hat{x}_k + \lambda_k (\hat{v}_k-\hat{u}_{k}).
\end{aligned}\right.
\end{equation}
By~\eqref{eq:assumption} and~\cite[Proposition~2.1]{dao2019resolvent}, the operators $A_{-q}^{(\theta,\sigma_A)}$, $B_{-q}^{(\theta,\sigma_B)}$ and
$T_{-q}^{(\theta, \sigma_T)}$ are maximally monotone,
%$(\theta\alpha_A+\sigma_A)$-monotone, maximally $(\theta\alpha_B+\sigma_B)$-monotone and maximally $(\theta\alpha_T+\sigma_T)$-monotone, respectively.
and by~\cite[Theorem~1(iii)]{franstrengthening}, $T_{-q}^{(\theta,\sigma_T)}$ is $\mu$-cocoercive. By assumption, $q \in \ran\left(\Id+\frac{\theta}{\sigma_A+\sigma_B+\sigma_T}\left(A+B+T\right)\right)$, and thus~\eqref{eq:assumption} and~\cite[Proposition~3]{franstrengthening} imply that
\begin{equation}\label{e:strength1}
\zer\left(A_{-q}^{(\theta,\sigma_A)}+ B_{-q}^{(\theta,\sigma_B)}+T_{-q}^{(\theta,\sigma_T)}\right) = \left\{ \frac{1}{\theta}\left( J_{\frac{\theta}{\sigma_A+\sigma_B+\sigma_T}(A+B+T)}(q)-q\right)\right\}.
\end{equation}
By Theorem~\ref{t:GradientDR}(\ref{it:GradientDR-ii}), $\hat{u}_k \wto \hat{u}$ and $\hat{v}_k \wto \hat{u}$, with
$$ \hat{u} \in \zer\left(A_{-q}^{(\theta,\sigma_A)}+ B_{-q}^{(\theta,\sigma_B)}+T_{-q}^{(\theta,\sigma_T)}\right),$$
and $\hat{x}_k \wto \hat{x}$, where $\hat{x}$ satisfies
\begin{equation}
\label{e:strength2}
\hat{u} = J_{\gamma A_{-q}^{(\theta,\sigma_A)}}(\hat{x}) \in \zer \left(A_{-q}^{(\theta,\sigma_A)}+ B_{-q}^{(\theta,\sigma_B)}+T_{-q}^{(\theta,\sigma_T)}\right).
\end{equation}
If $\theta \alpha_A+ \sigma_A >0$ (respectively $\theta \alpha_B+\sigma_B >0$), then $\hat{u}_k \to \hat{u}$ (respectively $\hat{v}_k \to \hat{u}$) by Theorem~\ref{t:GradientDR}(\ref{it:GradientDR-iii}), even if $\sum_{k\in \mathbb{N}}\lambda_k\left(2-\frac{\gamma}{2\mu}-\lambda_k\right) < +\infty$.
Thanks to~\cite[Proposition~2.1]{dao2019resolvent}, we may rewrite~\eqref{eq:DYstrengthened} as
\begin{equation*}
\left\{\begin{aligned}
\theta \hat{u}_k+q & = J_{\frac{\gamma\theta}{1+\gamma\sigma_A}A}\left(\frac{\theta}{1+\gamma\sigma_A}\hat{x}_k+q\right) \\
\theta \hat{v}_k+q & = J_{\frac{\gamma\theta}{1+\gamma\sigma_B}B}\left(\frac{\theta}{1+\gamma\sigma_B}\left(2\hat{u}_k-\hat{x}_k-\gamma \left(T(\theta\hat{u}_k+q)+ \sigma_T \hat{u}_k\right)\right)+q\right)
\end{aligned}\right.
\end{equation*}
Further, by~\eqref{e:strength2},~\eqref{e:strength1} and~\cite[Proposition~2.1]{dao2019resolvent},
$$J_{\frac{\theta}{\sigma_A+\sigma_B+\sigma_T}(A+B+T)}(q) =  \theta J_{\gamma A^{(\theta,\sigma_A)}_{-q}}\left(\hat{x}\right)+q=J_{\frac{\gamma \theta}{1+\gamma \sigma_A}A} \left(\frac{\theta}{1+\gamma\sigma_A}\hat{x}+q\right).$$
The result  follows by making the change of variables $(x_k,u_k,v_k) := (\theta\hat{x}_k+q,\theta\hat{u}_k+q,\theta\hat{v}_k+q)$ for all $k\in \mathbb{N}$ and $\bar{x} := \theta \hat{x}+ q$. The final assertion is a consequence of Remark~\ref{rem:1}(ii).
\end{proof}

\begin{remark}\label{rem:2}
Another way of computing the resolvent with parameter $\mu>0$ of $A+B+T$ at $q\in\Hi$ is applying the Davis--Yin splitting algorithm to $A$, $B$ and  $\widetilde{T}:= \frac{1}{\mu}(\Id -q)+ T$, where $\widetilde{T}$ is $\left(\beta^{-1}+\mu^{-1}\right)^{-1}$-cocoercive, by~\cite[Theorem~1(iii)]{franstrengthening}, and $\beta$ is the cocoercivity constant of $T$. Note that this is a particular instance covered by Theorem~\ref{t:DYsplitting}, taking $\sigma_T= \frac{1}{\mu}$, $\sigma_A=\sigma_B = 0$ and $\theta=1$.
\end{remark}

\section{Numerical experiments}\label{sec:4}

In this section we provide some numerical examples of the algorithms developed in the previous section.
These experiments aim not to be exhaustive and only intend to show the importance of appropriately choosing the stepsize and the relaxation parameters of the algorithms.

\subsection{A feasibility problem with hard and soft constraints}

Let $\mathbb{A},\mathbb{B},\mathbb{C} \subseteq \mathbb{R}^n$  be three closed and convex sets with nonempty intersection of the relative interiors of $\mathbb{A}$ and $\mathbb{B}$. Suppose $\mathbb{A}$ and $\mathbb{B}$ are \emph{hard constraints}, which need to be satisfied, and $\mathbb{C}$ is a third \emph{soft constraint}, which does not necessarily need to be fulfilled, but whose violation we want to reduce as much as possible. Imagine that, at the same time, we would like to find a point in $\mathbb{A}\cap \mathbb{B}$ as close as possible to a point $q\in\R^n$. This problem can be written as
\begin{equation}
\label{e:exstreng}
\argmin_{x \in \mathbb{A} \cap \mathbb{B}} \; \frac{1}{2}d^2(x,\mathbb{C}) + \frac{\rho}{2}\|x-q\|^2,
\end{equation}
where $d^2(x,\mathbb{C}):= \|x-P_{\mathbb{C}}(x)\|^2$ and $\rho >0$ is a regularization parameter specifying the importance of remaining close to the point $q$. Problem~\eqref{e:exstreng} can be reformulated as
\begin{equation*}
\argmin_{x \in \mathbb{R}^n} \iota_{\mathbb{A}}(x)+\iota_{\mathbb{B}}(x)+\frac{1}{2}\|x-q\|^2+\frac{1}{2\rho}d^2(x,\mathbb{C}),
\end{equation*}
whose solution is given by $\prox_{\left(\iota_{\mathbb{A}}+\iota_{\mathbb{B}}+\frac{1}{2\rho}d^2(\cdot,\mathbb{C})\right)}(q)$. The subdifferential sum rule (see, e.g.,~\cite[Corollary~16.50(v)]{bauschke2017}) guarantees the equality
$$
\prox_{\left(\iota_{\mathbb{A}}+\iota_{\mathbb{B}}+\frac{1}{2\rho}d^2(\cdot,\mathbb{C})\right)}(q) = J_{\left(\partial\iota_{\mathbb{A}}+\partial\iota_{\mathbb{B}}+\nabla\left(\frac{1}{2\rho}d^2(\cdot,\mathbb{C})\right)\right)}(q)
=J_{\left(N_{\mathbb{A}}+N_{\mathbb{B}}+\frac{1}{\rho}(\Id-P_\mathbb{C})\right)}(q),
$$
and thus, solving~\eqref{e:exstreng} boils down to computing the resolvent at $q$ of the sum of the three maximally monotone operators $A:=N_\mathbb{A}$, $B:=N_\mathbb{B}$ and $T:=\frac{1}{\rho}\left(\Id-P_\mathbb{C}\right)$, with $T$ being $\frac{1}{\rho}$-cocoercive  (see, e.g.,~\cite[Corollary~12.31]{bauschke2017}).

To illustrate on the problem~\eqref{e:exstreng}  the behavior of the Davis--Yin algorithm and its strengthened version derived in Theorem~\ref{t:DYsplitting}, we retake our simple introductory example of two balls $\mathbb{A}$ and $\mathbb{B}$ centered at $(-1.6,-0.75)$ and $(-0.35,0.12)$, with radii $0.55$ and $1$, respectively. We chose these values to make the problem slightly challenging. We now add a new third ball $\mathbb{C}$ with center $(1,-1)$ and radius $0.5$, the point $q:=(-1.75,1.5)$ and take $\rho:=1$. 
Observe that any combination of $\sigma_A\geq 0$, $\sigma_B\geq 0$ and $\sigma_T\geq 0$  such that $\theta:=\sigma_A+\sigma_B+\sigma_T>0$ satisfies the hypotheses of Theorem~\ref{t:DYsplitting}. Although finding the best values is beyond the scope of this work, for comparison, we tested the result of running the algorithm~\eqref{eq:DYsplitting} with $(\sigma_A,\sigma_B,\sigma_T)=(0,0,1/\mu)$ (which corresponds to Davis--Yin splitting, see Remark~\ref{rem:2}) and $(\sigma_A,\sigma_B,\sigma_T)=(0,1,1)$, using as starting point $x_0:=(0.7,1.7)$. In accordance with Theorem~\ref{t:DYsplitting}, the stepsize $\gamma$ must be chosen so that $\frac{\gamma}{\mu}\in{]0,4[}$, for $\mu=((\sigma_A+\sigma_B+\sigma_T)\rho+\sigma_T)^{-1}$. In Figure~\ref{fig:40} we have represented the iterates for $\lambda_k=0.99(2-\frac{\gamma}{2\mu})$ and for two values of $\frac{\gamma}{\mu}$, namely $1.5$ (overrelaxation) and $2.5$ (underrelaxation).

 \begin{figure}[ht!]\centering
\includegraphics[width=.49\textwidth]{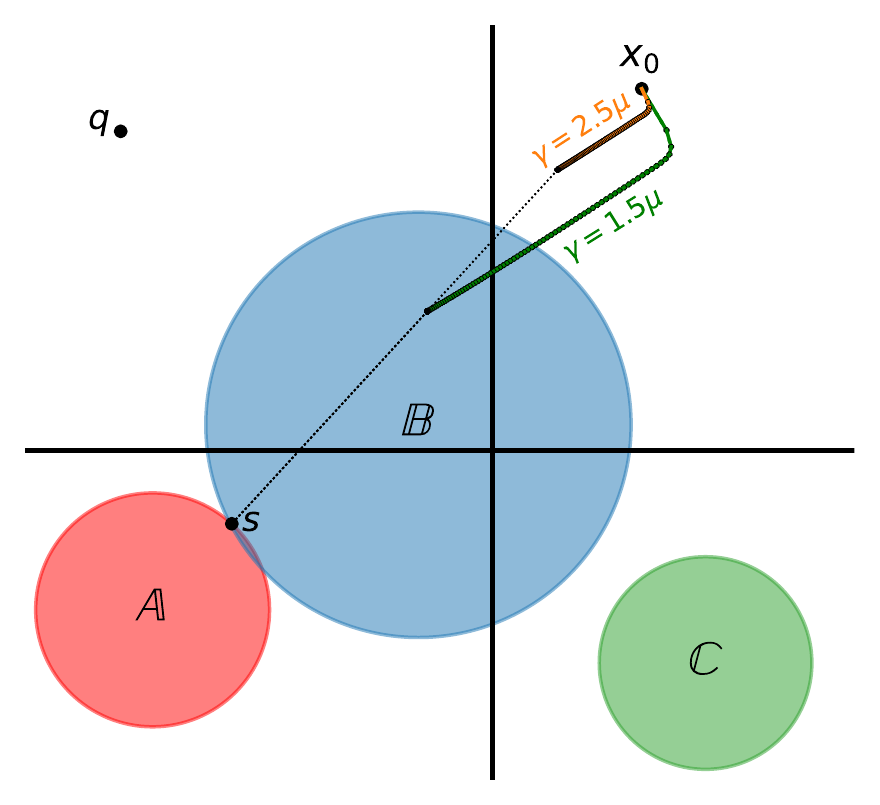}
\includegraphics[width=.49\textwidth]{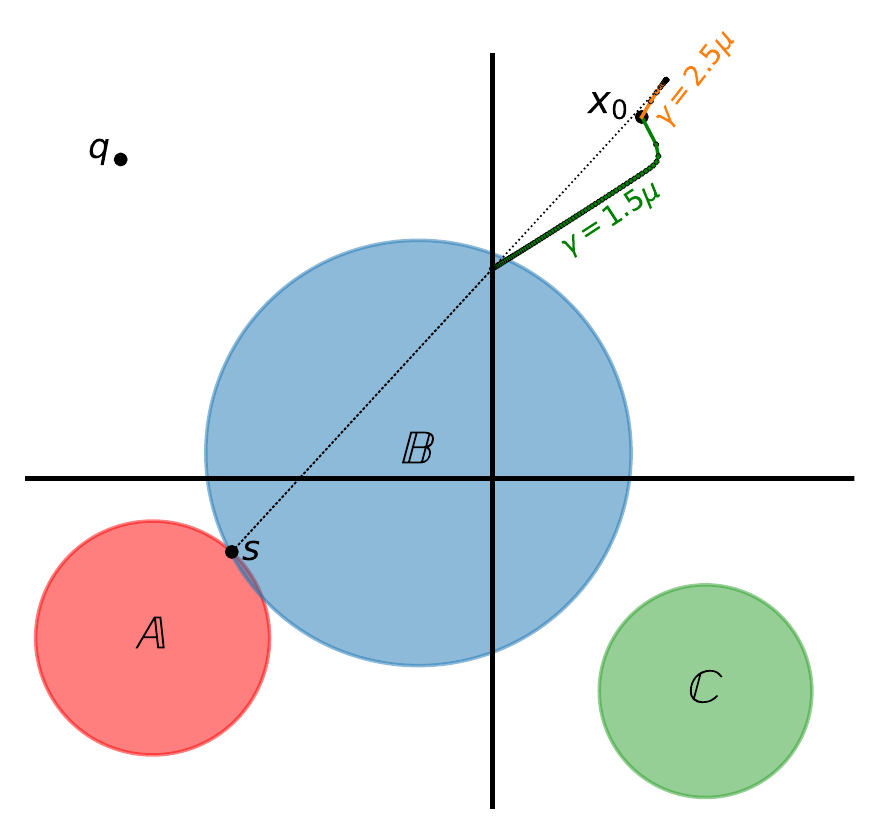}
\caption{Behavior of the iterates of the Davis--Yin (left) and the strengthened-Davis--Yin (right) splitting algorithms for the problem~\eqref{e:exstreng} for two stepsize parameters~$\gamma$ and $\lambda_k=0.99(2-\gamma/(2\mu))$. Since $\sigma_A=0$, the solution is obtained after projecting the fixed point onto the set $\mathbb{A}$.}\label{fig:40}
\end{figure}

In order to obtain the best combination of the stepsize and relaxation parameters, we run the algorithms for every possible value of $(\frac{\gamma}{\mu},\lambda)$  on a grid with 4950 points in $]0,4[\times]0,2[$. The algorithms were stopped when the norm of the difference between the shadow sequence $P_A(x_k)$ and the solution to the problem was smaller than $10^{-8}$. The solution, which is approximately equal to $(-1.227559, -0.3452923)$, was computed in Maple by numerically solving the KKT conditions with high precision. A contour plot representing the number of iterations is shown in Figure~\ref{fig:4}. The minimum number of iterations for Davis--Yin was 17 and it was attained at $(\frac{\gamma}{\mu},\lambda)=(3.11,0.43)$, and for the strengthened-Davis--Yin was 16 and it was reached at three pair of values of $\frac{\gamma}{\mu}$ and $\lambda$, namely $\frac{\gamma}{\mu}=2.34,\lambda\in\{0.79,0.81\}$ and $\frac{\gamma}{\mu}=2.39,\lambda=0.79$.

\begin{figure}[ht!]\centering
\includegraphics[width=.49\textwidth]{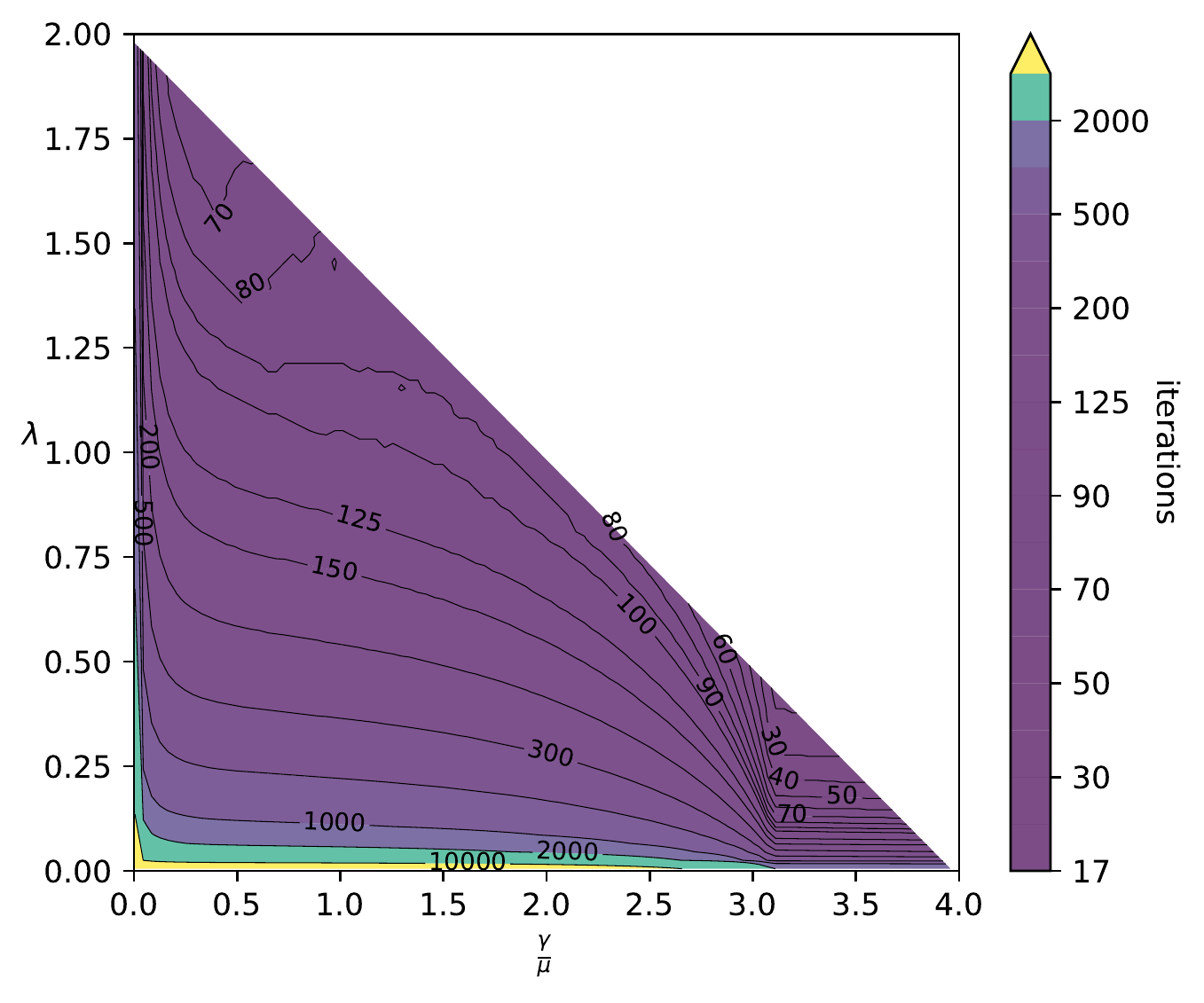}
\includegraphics[width=.49\textwidth]{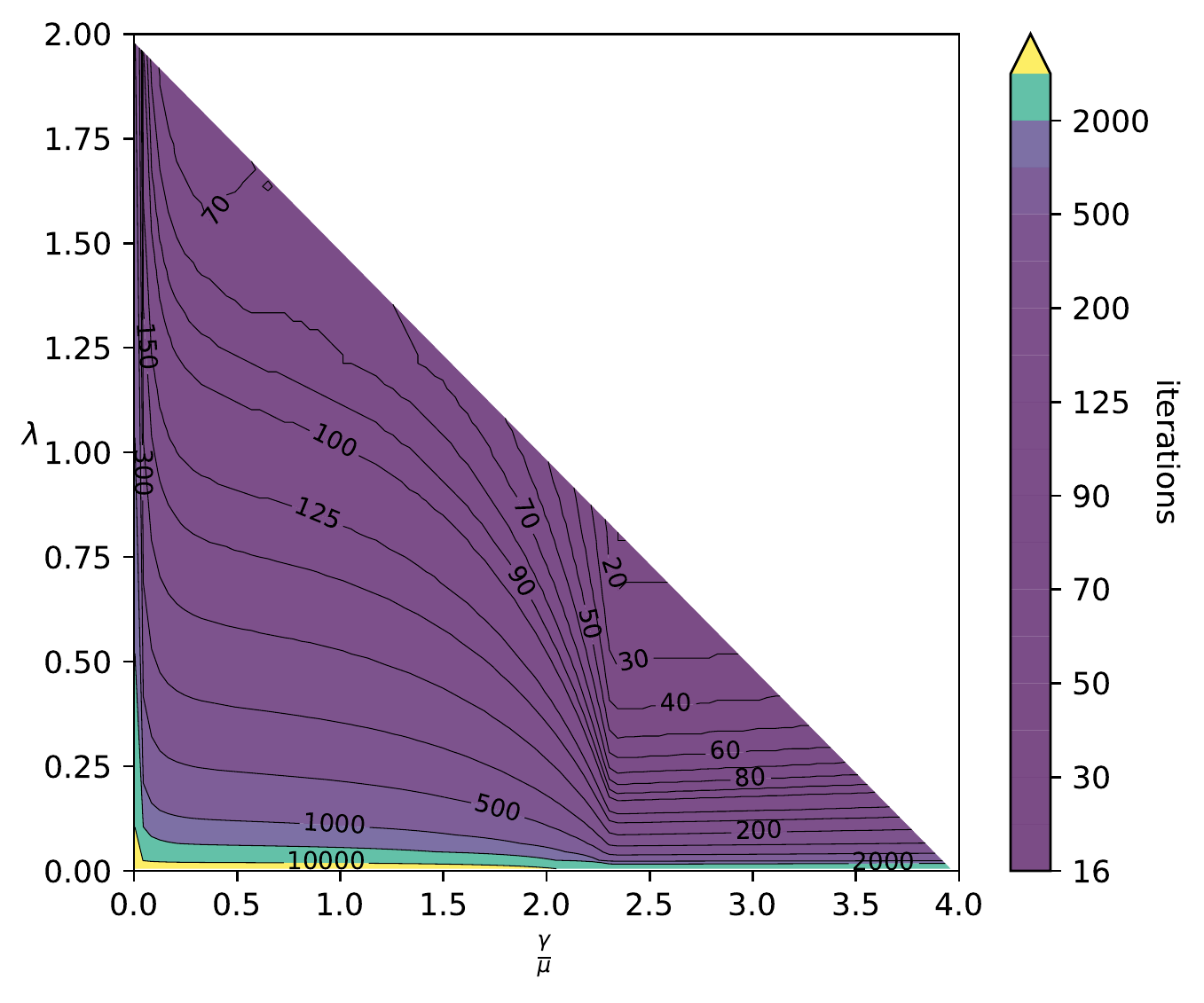}
\caption{Number of iterations needed until the shadow sequence is sufficiently close to the solution $s$  when the Davis--Yin (left) and the strengthened-Davis--Yin (right) splitting algorithms  are applied for different values of $\gamma$ and $\lambda_k=\lambda$, with the experiment setting shown in Figure~\ref{fig:40}.}\label{fig:4}
\end{figure}

\subsection{Image recovery via $\ell_1$ regularization}

The restoration of blurred images using $\ell_1$ regularization has become a standard application in the literature to test the performance of forward-backward algorithms, see~\cite{fista}. This consists in solving a minimization problem of the form
\begin{equation}
\label{e:exp1}
\argmin_{x \in \mathbb{R}^n} \; \mu \|x \|_1 + \frac{1}{2}\| Mx -b\|^2_2,
\end{equation}
where $M \in \mathbb{R}^{m \times n}$, $b \in \mathbb{R}^m$ is the observed blurred image (the vectorization of the two-dimensional matrix) and $\mu > 0$ is a regularization parameter. Setting $B= \partial\left(\mu \|\cdot\|_1\right)$ and $T=M^{T}(Mx-b)$, this problem can be reformulated as finding a zero of the sum $B+T$ of two maximally monotone operators. Since $T$ is Lipschitz continuous, we can employ the forward-backward algorithm (i.e., Davis--Yin with $A=0$), to solve~\eqref{e:exp1}. Note that the proximity operator of the $\ell_1$-norm is the well-known soft thresholding function from~Example~\ref{ex:prox}.
As pixel values must be in $[0,1]$, it is more realistic to solve instead the problem
\begin{equation*}
\label{e:exp2}
\argmin_{x \in [0,1]^n} \; \mu \|x \|_1 + \frac{1}{2}\| Mx -b\|^2_2,
\end{equation*}
Setting $A = N_{[0,1]^n}$ and $B$ and $T$ as above, this problem can be solved without much additional effort using the Davis--Yin splitting algorithm.

For our tests we replicated the wavelet-based restoration method in~\cite[Section~5.1.]{fista}, including the additional constraint $x\in[0,1]^n$. We also ran our experiments without this constraint (applying thus forward-backward) and the results were basically the same, so we do not include them for brevity. We employed as observed images the widely-used $256 \times 256$ pixels cameraman image and a picture of a symbol from the University of Alicante: the sculpture ``Dibuixar l'espai'' (by Pepe Azor\'in), with a resolution of $600\times 800$ pixels. The images, shown in Figure~\ref{fig:3}, were subjected to a Gaussian $9 \times 9$ blur with standard deviation 4, followed by an additive zero-mean Gaussian noise with standard deviation $10^{-3}$. We chose $M=RW$, where $R$ is the matrix representing the blur operator and $W$ is the inverse of the three stage Haar wavelet transform. The regularization parameter was taken as $\mu=2\cdot 10^{-5}$. The Lipschitz constant of $T$ is the spectral radius of $M^TM$, which is equal to $1$. Thus, $T$ is 1-cocoercive and the stepsize in the Davis--Yin algorithm can be chosen in the interval $]0,4[$. For values of $(\gamma,\lambda)$ on a grid with 4950 points in $]0,4[\times]0,2[$, we performed 200 iterations of the algorithm taking as initial image the observed blurred image. Figure~\ref{fig:5} shows the value of the objective function in the final iteration. We observe a symmetry with respect to the diagonal. The lowest values of the objective function were $0.349$ for the cameraman and $2.684$ for the sculpture, and they were both attained at $(\gamma,\lambda)=(1.98, 0.99)$.

\begin{figure}[htb!]\centering

\includegraphics[width=.33\textwidth]{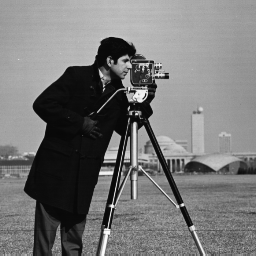}\hfill%
\includegraphics[width=.33\textwidth]{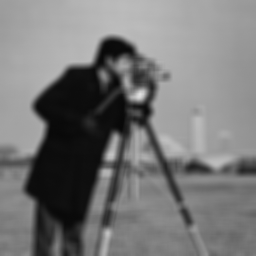}\hfill%
\includegraphics[width=.33\textwidth]{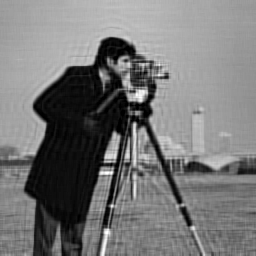}\\
\includegraphics[width=.33\textwidth]{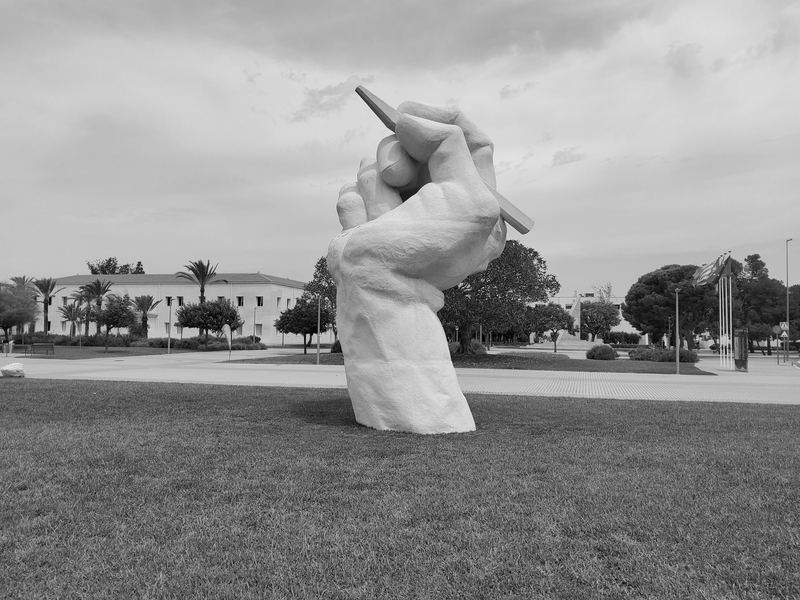}\hfill%
\includegraphics[width=.33\textwidth]{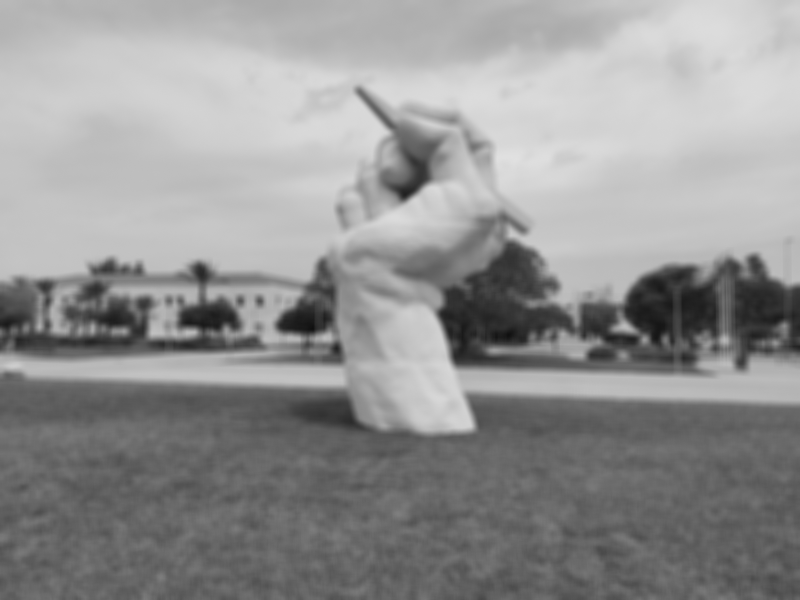}\hfill%
\includegraphics[width=.33\textwidth]{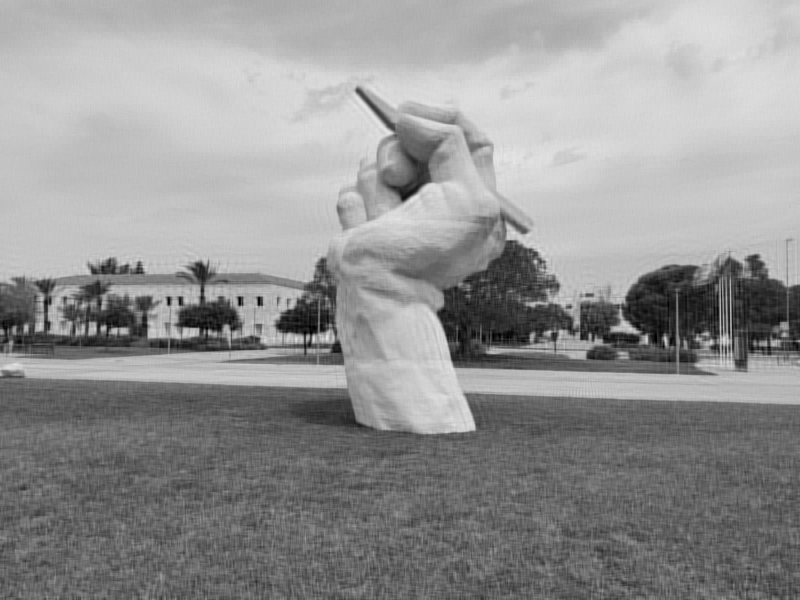}%
\caption{Original (left), observed blurred (middle) and restored (right) images, showing the cameraman at the top and the sculpture ``Dibuixar l'espai'' at the bottom. The Davis--Yin algorithm was applied for 200 iterations with $\gamma=1.98$ and $\lambda=0.99$, using as starting point the observed blurred image. }\label{fig:3}
\end{figure}

\begin{figure}[ht!]\centering
\includegraphics[width=.49\textwidth]{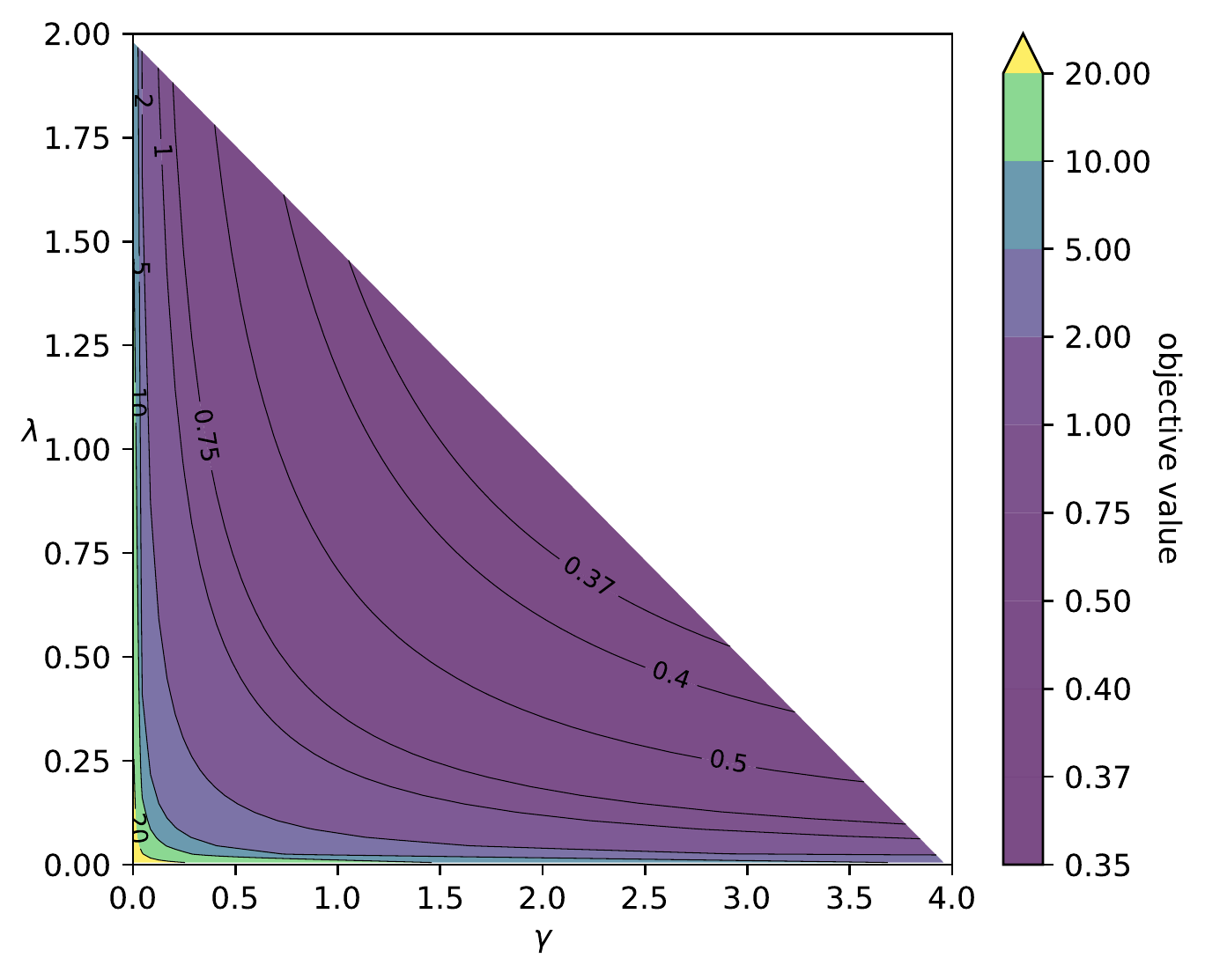}
\includegraphics[width=.49\textwidth]{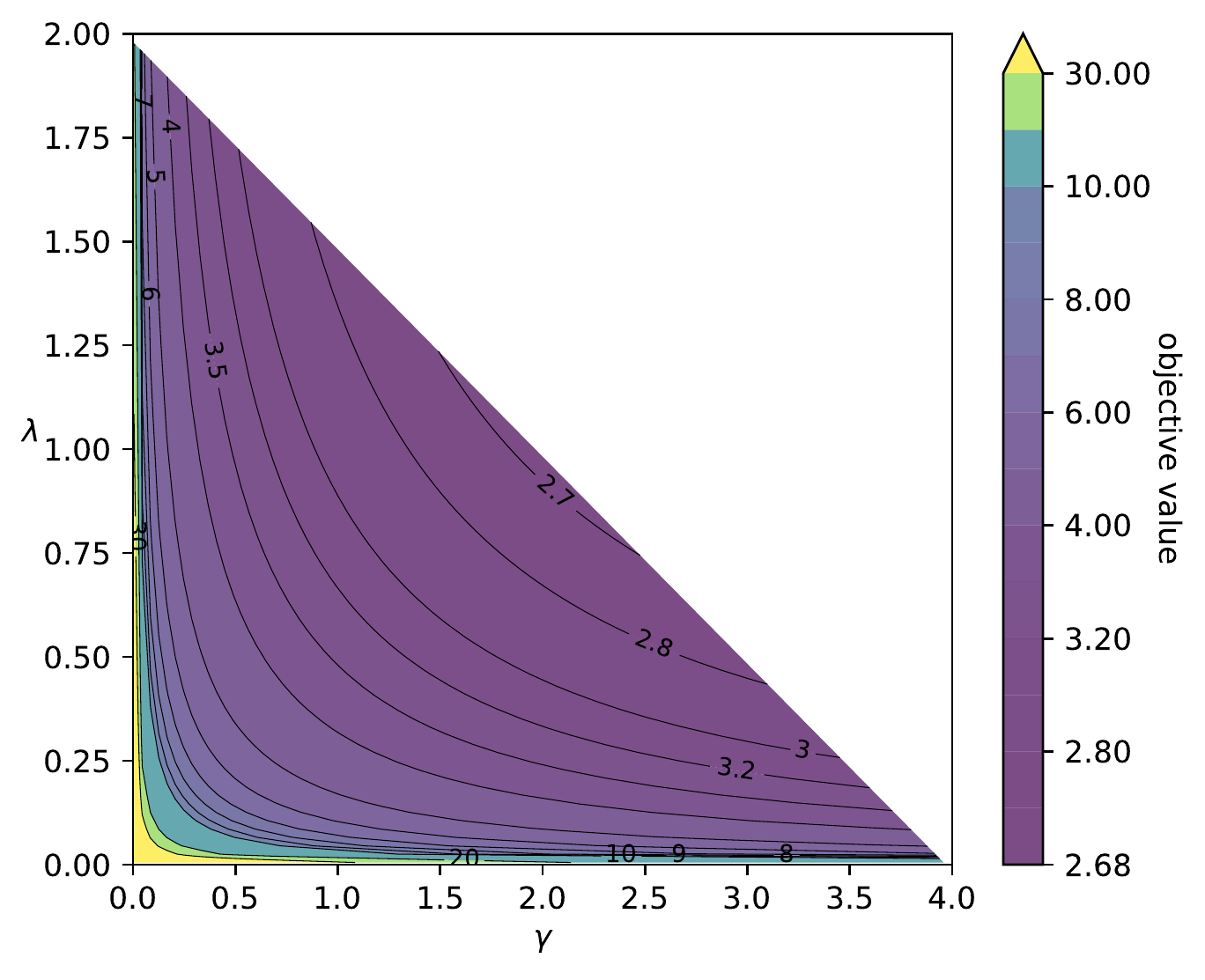}
\caption{Objective function value after 200 iterations of the forward-backward algorithm applied to the cameraman (left) and the sculpture ``Dibuixar l'espai'' (right), for different values of $\gamma$ and $\lambda$, and taking as  starting point the observed blurred image. }\label{fig:5}
\end{figure}

\section{Conclusions}\label{sec:5}
We have presented an alternative proof of convergence for the Davis--Yin splitting algorithm without requiring the Davis--Yin operator~\eqref{eq:DYoperator} to be averaged. The proof was solely based on monotone operator theory and has the additional advantage of allowing larger stepsizes, up to four times the cocoercivity constant of the single-valued operator, doubling thus the range of values allowed in~\cite{davis2017three}. As a consequence, the same conclusion applies to the forward-backward splitting algorithm. We have also derived a \emph{strengthened} version of the algorithm for computing the resolvent of the sum, based on the framework developed in~\cite{franstrengthening}. The numerical experiments included show the importance of appropriately  selecting the stepsize and relaxation parameters. In most of our tests, the behavior of the algorithm with respect to the parameters was symmetric, as the one shown in Figure~\ref{fig:5}. Selecting the best parameters is not a simple task, but even so, it is clear that having more freedom in the choice of the stepsize parameter can only be advantageous.

\paragraph{{\small Acknowledgements}}{\small
\hspace{-2mm}
The authors would like to thank Patrick Combettes for making us aware of~\cite{DaoPhan21} right before submitting this work. We thank two anonymous referees for their careful reading and their constructive comments which helped improve our manuscript.\smallskip

FJAA and DTB were partially supported by the Ministry of Science, Innovation and Universities of Spain and the European Regional Development Fund (ERDF) of the European Commission, Grant PGC2018-097960-B-C22. FJAA was partially supported by the Generalitat Valenciana (AICO/2021/165). DTB was supported by MINECO and European Social Fund (PRE2019-090751) under the program ``Ayudas para contratos predoctorales para la formaci\'{o}n de doctores''
2019.}

\end{document}